%% file: multizd.tex
\documentclass{amsart}
\usepackage[numbers,square]{natbib}
\usepackage{amsfonts,amsthm,amsmath,amssymb}
\usepackage{mathtools,verbatim,mathrsfs}
\usepackage{enumitem}
\usepackage{hyperref}

\newcommand*\mcup{\mathbin{\mathpalette\mcupinn\relax}}
\newcommand*\mcupinn[2]{\vcenter{\hbox{$\mathsurround=0pt
  \ifx\displaystyle#1\textstyle\else#1\fi\bigcup$}}}

\usepackage{tikz,tikz-cd}
\usetikzlibrary{matrix, calc, arrows}

\numberwithin{equation}{section}

\input multizd_header.tex

\author[V. Bergelson]{Vitaly Bergelson}
\thanks{The first author gratefully acknowledges the support of the NSF under grant DMS-1500575.}
\author[D. Glasscock]{Daniel Glasscock}
\address{Department of Mathematics, The Ohio State University, 231 W. 18th Ave., Columbus, OH 43210}
\email{vitaly@math.osu.edu and glasscock.4@math.osu.edu}

\title[Multiplicative richness of additively large sets]{Multiplicative richness of additively large sets in $\Zd$}

\begin{document}

\begin{abstract}
In their proof of the IP Szemer\'edi theorem, a far reaching extension of the classic theorem of Szemer\'edi on arithmetic progressions, Furstenberg and Katznelson \cite{furstenbergkatznelsonipszem} introduced an important class of additively large sets called $\ip_r^*$ sets  which underlies recurrence aspects in dynamics and is instrumental to enhanced formulations of combinatorial results. The authors recently showed that additive $\ip_r^*$ subsets of $\Z^d$ are multiplicatively rich with respect to every multiplication on $\Zd$ without zero divisors (e.g. multiplications induced by degree $d$ number fields). In this paper, we explain the relationships between classes of multiplicative largeness with respect to different multiplications on $\Zd$. We show, for example, that in contrast to the case for $\Z$, there are infinitely many different notions of multiplicative piecewise syndeticity for subsets of $\Z^d$ when $d \geq 2$. This is accomplished by using the associated algebra representations to prove the existence of sets which are large with respect to some multiplications while small with respect to others. In the process, we give necessary and sufficient conditions for a linear transformation to preserve a class of multiplicatively large sets. One consequence of our results is that additive $\ip_r^*$ sets are multiplicatively rich in infinitely many genuinely different ways. We conclude by cataloging a number of sources of additive $\ip_r^*$ sets from combinatorics and dynamics.
\end{abstract}

\maketitle

\section{Introduction}\label{sec:intro} 

\subsection{Motivation}

A subset of $\Zd$ is called \emph{\adly{} $\ipr$}, $r \in \N$, if it contains a set of the form
\begin{align}\label{eqn:defoffinitesumsset}\text{FS}(x_1, \ldots, x_r) = \left\{ \sum_{i \in I} x_i \ \middle| \ \emptyset \neq I \subseteq \{1, \ldots, r\} \right\}, \quad x_1, \ldots, x_r \in \Zd.\end{align}
A subset of $\Zd$ is called \emph{\adly{} $\ipr^*$} if it has non-empty intersection with every \adive{} $\ipr$ set in $\Zd$.

\Adive{} $\ipr^*$ sets first appeared in the work of Furstenberg and Katznelson \cite{furstenbergkatznelsonipszem} on the IP multidimensional Szemer\'{e}di theorem, and they appeared implicitly in \cite[Section 1]{bergelsonleibmanams} in connection with the multidimensional polynomial van der Waerden theorem. Loosely speaking, given $A \subseteq \Zd$ with positive upper Banach density (resp. a finite partition of $\Zd$) and $\ell \in \N$, there exists $r \in \N$ for which the set of mesh sizes $m \in \N$ of finite lattices $z+ \{m, 2m, \ldots, \ell m\}^d$ contained in $A$ (resp. contained in some cell of the partition) is \adly{} $\ipr^*$ in $\N$.

The main results in \cite{furstenbergkatznelsonipszem} pertaining to measurable multiple recurrence and in \cite[Section 1]{bergelsonleibmanams} pertaining to topological polynomial multiple recurrence are formulated in terms of \adive{} $\ip^*$ sets -- the infinitary analogue of $\ipr^*$ -- with the precise role of $\ipr^*$ sets explained in final remarks in the former and implicit in the proofs in the latter. An example of the strength of the $\ipr^*$-formulations was demonstrated in \cite{blzpaper}, where it is shown that the set of prime numbers minus one, $\mathbb{P}-1$, is \adly{} $\ipr$ for all $r \in \N$, and hence that there exist finite lattices in the set $A$ (or in some cell of the partition) mentioned above with mesh size one less than a prime. (Their work also shows that the same result holds for $\mathbb{P}+1$ but it cannot hold for other translate of $\mathbb{P}$.) More general $\ipr^*$-formulations of main theorems from \cite{furstenbergkatznelsonipszem} and \cite{bergelsonleibmanams} can be found in \cite[Section 6]{BGpaperone}.

Sets of return times in measure theoretical and topological dynamics underpin the $\ipr^*$ structure in the results of the previously mentioned works and provide a good source of concrete examples of \adive{} $\ipr^*$ sets. For example, the first author made use of the Hales-Jewett theorem to prove in \cite{BergelsonSurveytwoten} that for all $f \in \R[x]$ with $f(0)=0$ and all $\eps > 0$, there exists $r \in \N$ such that the set $\{n \in \Z \ | \ \|f(n)\| < \eps \}$ is \adly{} $\ipr^*$, where $\| \cdot \|$ denotes the distance to the nearest integer. The polynomial Hales-Jewett theorem \cite{bergelsonleibmanannals} was leveraged recently in \cite{BLiprstarcharacterization} to give far-reaching generalizations; as an example, there exists $r \in \N$ for which the set
\begin{align}\label{eqn:impressiveexample} \Big\{(n,m) \in \Z^2 \ \Big| \ \big\| \sqrt{2} m^3 \big[ \sqrt[4]{5} n^6 + \pi n m^7 [ \sqrt[8]{9} n^{10} ] \big] \big\| < \eps \Big\},\end{align}
is \adly{} $\ipr^*$ in $\Z^2$. This is a consequence of a much more general result from \cite{BLiprstarcharacterization}: return times of a point to a neighborhood of itself in a nilsystem\footnote{A nilsystem is a system of the form $(X,T)$ where $X$ is a compact
homogeneous space of a nilpotent Lie group $G$ and $T$ is a translation of $X$ by an element of $G$.} is an \adive{} $\ipr^*$ set in $\Z$ and, moreover, this recurrence property characterizes so-called pre-nilsystems.

\Adive{} $\ipr^*$ sets are ``large'' in many senses. They possess a filter property: given $r_1$ and $r_2$ in $\N$, there exists $r_3 \in \N$ such that the intersection of any $\iprset {r_1}^*$ set with any $\iprset {r_2}^*$ set is an $\iprset{r_3}^*$ set; see \cite[Proposition 2.5]{BRcountablefields}. \Adive{} $\ipr^*$ sets are ``additively large:'' they have lower density bounded from below by $2^{1-r}$; see \cite[Theorem 10.4]{furstenbergkatznelsonipszem}. \Adive{} $\ipr^*$ sets are also ``additively rich:'' they are syndetic (have bounded gaps) and central (see Definition \ref{def:defofcentral}), and hence they contain very many solutions to any partition regular system of linear equations; see \cite[Chapter 9]{furstenberg-book}. Perhaps more surprisingly, \adive{} $\ipr^*$ sets in $\N$ are also ``multiplicatively large:'' they have non-empty intersection with every multiplicatively central set in $\N$; see \cite[Theorem 3.5]{berghind-onipsets}. This implies that any \adive{} $\ipr^*$ set $A \subseteq \N$ is \emph{multiplicatively syndetic}: there exists a finite set $F \subseteq \N$ such that for all $n \in \N$, $(F \cdot n) \cap A \neq \emptyset$.

It is important to note the ways in which additive $\ipr^*$ sets are larger than their infinitary analogues, additive $\ip$ sets. A subset of $\Zd$ is called \emph{\adly{} $\ip$} if it contains a set of the form
\begin{align}\label{eqn:defofsumsset}\text{FS}(x_n)_{n=1}^\infty = \left\{ \sum_{i \in I} x_i \ \middle| \ \text{finite, non-empty } I \subseteq \N \right\}, \quad (x_n)_{n=1}^\infty \subseteq \Zd,\end{align}
and it is called an \emph{\adive{} $\ip^*$} set if it has non-empty intersection with every \adive{} $\ip$ set in $\Zd$. Like $\ipr^*$ sets, additive $\ip^*$ sets possess a filter property (this follows from the classic theorem of Hindman \cite{hindmanoriginal}) and are additively syndetic and central. In contrast to $\ipr^*$ sets, however, additive $\ip^*$ sets need not be multiplicatively syndetic, as was shown for subsets of $\N$ in \cite[Theorem 3.6]{berghind-onipsets}. We improve on this result with Theorem \ref{thm:addipdoesnotimplymultsyndetic} by showing that there are additive $\ip^*$ sets in $\Zd$ which are not multiplicatively syndetic with respect to any proper multiplication (defined in the next paragraph) on $\Z^d$.

A binary operation $\tone: \Zd \times \Zd \to \Zd$ is called a \emph{proper multiplication on $\Zd$} if it makes the abelian group $(\Zd,+)$ into a (not-necessarily unital or commutative) ring $(\Zd,+,\tone)$ without zero divisors. The family of rings $\big\{\Z\big[\sqrt c \big] \ | \ c \in \Z \text{ not a square} \big\}$ yield concrete examples of proper multiplications on $\Zt$ (under the usual identification of elements of these rings with $\Zt$). The Lipschitz integral quaternions
\[\quatz = \big\{ x_1 + ix_2 + jx_3 + kx_4 \ \big| \ x_1, x_2, x_3, x_4 \in \Z \big\} \]
provide an example of a non-commutative proper multiplication on $\Z^4$. Since a proper multiplication is determined by the pairwise products of the elements of a basis, there are only countably many proper multiplications on $\Zd$.

Let $\tone$ be a proper multiplication on $\Zd$. A subset $A \subseteq \Zdz = \Zd \setminus \{0\}$ is called \emph{\psstar{} with respect to $\tone$} if for all finite $F \subseteq \Zdz$, there exists a finite $G \subseteq \Zdz$ such that for all $x \in \Zdz$, there exists $g \in G$ for which $F \tone g \tone x \subseteq A$. (This is equivalent to the set $A$ having non-empty intersection with all subsets of $\Zdz$ which are piecewise syndetic with respect to $\tone$; see Definition \ref{def:syndthick} and Lemma \ref{lem:altcharofpstar}). The authors recently proved the following theorem, improving on \cite[Theorem 3.5]{berghind-onipsets}.

\begin{theorem}[{\cite[Corollary 6.3]{BGpaperone}}]\label{thm:thmfrombgotherpaper}
Let $A \subseteq \Zd$ be \adly{} $\ipr^*$. The set $A \setminus \{0\}$ is multiplicatively \psstar{} with respect to every proper multiplication on $\Zd$.
\end{theorem}

Thus, in the sense of Theorem \ref{thm:thmfrombgotherpaper}, the set defined in (\ref{eqn:impressiveexample}) above is multiplicatively large with respect to all proper multiplications on $\Zt$; in particular, it is multiplicatively large with respect to all of the multiplications induced on $\Z^2$ from the rings $\big\{\Z[\sqrt c] \ | \ c \in \Z \text{ not a square} \big\}$. In an effort to understand the implications of a set being multiplicatively large with respect to \emph{all} of these multiplications, one is led naturally to ask about the relationships between the various notions of ``multiplicatively \psstar{}'' induced by each.

More generally, one is led to ask about the relationship amongst notions of multiplicative largeness for subsets of $\Zdz$ with respect to the various proper multiplications on $\Zd$. This question motivates the main results in this paper.

\subsection{Statement of results}
Let $\tone$ be a proper multiplication on $\Zd$, and denote by $\pws^*(\tone)$ the class of subsets of $\Zdz$ which are \psstar{} with respect to $\tone$. The following theorem sheds light on the statement ``\psstar{} with respect to every proper multiplication'' appearing in Theorem \ref{thm:thmfrombgotherpaper} by describing how the classes $\pws^*(\tone)$ and $\pws^*(\ttwo)$ relate for different proper multiplications $\tone$ and $\ttwo$ on $\Zd$. 

\begin{mainthma}\label{thm:maintheorema}
Let $\tone, \ttwo$ be proper multiplications on $\Zd$. The classes $\pws^*(\tone)$ and $\pws^* (\ttwo)$ are equal if and only if there exist $v, w \in \Zdz$ such that for all $x, y \in \Zdz$,
\begin{align}\label{eqn:mainconditiontheorema}x \tone v \tone y = x \ttwo w \ttwo y.\end{align}
\end{mainthma}

We will prove that Theorem A holds not only for the class $\pws^*$, but for many other classes of multiplicatively large sets: syndetic sets, piecewise syndetic sets, and sets with positive upper Banach density, among others. This is accomplished in two steps. First, we interpret the condition in (\ref{eqn:mainconditiontheorema}) in terms of the images of $\Zd$ under the representations of the multiplications involved. Then, we use this condition to construct sets which are ``large'' (multiplicatively thick) with respect to a given set of proper multiplications and ``small'' (not multiplicatively piecewise syndetic) with respect to others; see Lemma \ref{lem:equivconditionsforaligned} and Theorem \ref{thm:thickbutwithoutdensity}.

In the course of proving Theorem A, we will show that the classes $\pws^*(\tone)$ and $\pws^*(\ttwo)$ either coincide or are in general position. Because \adive{} $\ipr^*$ sets belong to both of these classes, this result furthers our understanding of the multiplicative largeness of \adive{} $\ipr^*$ sets. We will show, for example, that no two of the multiplications on $\Zt$ arising from the rings $\Z[\sqrt c]$, $c \in \Z$ not a square, satisfy (\ref{eqn:mainconditiontheorema}), and thereby show that the set in (\ref{eqn:impressiveexample}) is multiplicatively large in countably many distinct ways.

A subset of $\Zdz$ is called \emph{multiplicatively $\ipr$ with respect to $\tone$} if it contains a set of the form in (\ref{eqn:defoffinitesumsset}) where addition is replaced by $\tone$ and the indices in the product are taken in increasing order (see Definition \ref{def:ipstructure}). Denote by $\iprclass^*(\tone)$ the class of subsets of $\Zdz$ which have non-empty intersection with all subsets of $\Zdz$ which are multiplicatively $\ipr$ with respect to $\tone$. Denote by $\toneop$ the opposite operation of $\tone$, defined by $x \toneop y = y \tone x$.

\begin{mainthmb}\label{thm:maintheoremb}
Let $\tone, \ttwo$ be proper multiplications on $\Zd$. For all $r \geq 2$, the classes $\iprclass^*(\tone)$ and $\iprclass^*(\ttwo)$ are equal if and only if $\tone = \ttwo$ or $\tone = \ttwoop$.
\end{mainthmb}

We will prove in addition to Theorem B a version of it for multiplicative $\ip^*$ sets (defined in analogy with (\ref{eqn:defofsumsset})), showing that a proper multiplication on $\Zd$ is uniquely determined by its family of multiplicative $\ip^*$ sets. This is accomplished by constructing sets which are multiplicatively $\ipset$ with respect to $\tone$ which do not contain solutions to the equation $x \ttwo y = z$.

As a corollary of Theorem A, we derive necessary and sufficient conditions for a $\Z$-linear transformation $T: \Zd \to \Zd$ to preserve the class of \psstar{} sets with respect to a proper multiplication $\tone$ on $\Zd$: $T \big( \pws^*(\tone) \big) \subseteq \pws^*(\tone)$. Just as for the previous results, we will prove several variants of this theorem for various classes of multiplicative largeness.

\begin{maincorc}
Let $\tone$ be a proper multiplication on $\Zd$, and let $T: \Zd \to \Zd$ be $\Z$-linear with non-zero determinant. The map $T$ preserves the class $\pws^*(\tone)$ if and only if for all $x \in \Zdz$, there exists a non-zero $c \in \Z$ and $w \in \Zdz$ such that for all $y \in \Zdz$,
\begin{align}\label{eqn:mainconditioncorollaryc}cT(x \tone y) = w \tone (Ty).\end{align}
\end{maincorc}

Reformulating (\ref{eqn:mainconditioncorollaryc}) in terms of the ring representation of $(\Zd,+,\tone)$ will allow us, in many cases, to write down explicitly the set of matrices satisfying it. Understanding the collection of transformations that preserve the property of being a \psstar{} set with respect to a particular multiplication provides a geometric perspective on this notion of multiplicative largeness.

Theorem \ref{thm:thmfrombgotherpaper} gives that \adive{} $\ipr^*$ sets in $\Zd$ are multiplicatively \psstar{} with respect to all proper multiplications on $\Zd$, and Theorem A shows that the classes of multiplicatively \psstar{} sets for various proper multiplications on $\Zd$ are, predominantly, in general position. Thus, the results in this paper serve to enhance the conclusions of those results which yield \adive{} $\ipr^*$ sets in $\Zd$. One natural source of such sets comes from Diophantine approximation, as was evidenced above by the set defined in (\ref{eqn:impressiveexample}). More generally, the modulo 1 return times to zero of any \emph{constant-free generalized polynomial} form an additive $\ipr^*$ set (see Section \ref{sec:applications}).

Times of multiple recurrence of sets of positive density provide another source of $\ipr^*$ sets. Given a set $A \subseteq \Z^d$ of positive additive upper Banach density and $\Z$-linear transformations $T_1, \ldots, T_k: \Z^d \to \Z^d$, there exists $r \in \N$ for which the set
\[R = \big\{ z \in \Z^d \ \big | \ (A - T_1 z) \cap \cdots \cap (A- T_k z) \neq \emptyset \big\}\]
is \adly{} $\ipr^*$; see Theorem \ref{thm:multiszem}. This is an enhanced multidimensional version of the classic Szemer\'edi theorem on arithmetic progressions due to Furstenberg and Katznelson \cite{furstenbergkatznelsonipszem}. By Theorem \ref{thm:thmfrombgotherpaper}, the set $R \setminus \{0\}$ is multiplicatively \psstar{} with respect to all proper multiplications on $\Zd$. In particular (Lemma \ref{lem:basicfactsofdensity} \ref{item:combcharofpwsstar}), for all proper multiplications $\tone$ on $\Zdz$ and for all finite sets $F \subseteq \Zdz$, there exists a multiplicatively syndetic set $S \subseteq \Zdz$ for which $F \tone S \subseteq R$.

The paper is organized as follows. We begin in Section \ref{sec:defs} by defining the relevant notions of largeness for subsets of semigroups. In Section \ref{sec:ringreps}, we develop the algebra necessary for the proofs of the main results, and in Section \ref{sec:earlyexamples}, we give concrete examples of proper multiplications and the spaces associated with their representations. Proofs of Theorems A, B, and Corollary C appear in sections \ref{sec:manymults}, \ref{sec:maintheoremtwo}, and \ref{sec:proofofcorollary}, respectively. We conclude in Section \ref{sec:applications} by giving a combinatorial characterization of $\pws^*$ and cataloging several sources of \adive{} $\ipr^*$ sets.

\subsection{Acknowledgements} The authors would like to thank Daniel Shapiro for useful discussions regarding $\Q$-algebras.  Thanks is also due to Rostislav Grigorchuk and Mark Sapir for their help with Lemma \ref{lem:nonamenable}.

\section{Classes of largeness in semigroups}\label{sec:defs}  

In this section, we define several notions of largeness for subsets of semigroups. The best general references are \cite[Chapter 9]{furstenberg-book} and \cite[Section 1]{BHabundant}, though we will avoid the machinery of ultrafilters in this paper. While the results in this section are stated for a general semigroup $\semistimes$, there are two semigroups in particular to keep in mind: $(\Zd,+)$ and $(\Zdz,\tone)$, where $\Zdz = \Zd \setminus \{0\}$ and $\tone$ is a proper multiplication on $\Zd$ (see the beginning of Section \ref{sec:ringreps}).

Denote by $\N$ the set of natural numbers $\{1, 2, \ldots\}$. For $S$ a set, denote by $\finitesubsets S$ and $\subsets S$ the collections of all finite subsets (including the empty set) and all subsets of $S$, respectively. For $\semistimes$ a semigroup, $x \in S$, and $A \subseteq S$, let $x\tgen A$ denote $\{x \tgen a \ | \ a \in A \}$ and $x^{-1}\tgen A$ denote $\{s \in S \ | \ x \tgen s \in A \}$; the right handed versions $A \tgen x$ and $A \tgen x^{-1}$ denote the right handed analogues.

\begin{definition}\label{def:syndthick}
Let $\semistimes$ be a semigroup and $A \subseteq S$.
\begin{enumerate}[label=(\Roman*)]
\item $A$ is (right) \emph{syndetic} if there exist $s_1, \ldots, s_k \in S$ such that $S = s_1^{-1} \cdot A \cup \cdots \cup s_k^{-1} \cdot A$.
\item $A$ is (left) \emph{thick} if for all $F \in \finitesubsets S$, there exists $x \in S$ for which $F \cdot x \subseteq A$.
\item $A$ is (right) \emph{piecewise syndetic} if there exist $s_1, \ldots, s_k \in S$ such that the set $s_1^{-1}\cdot A \cup \cdots \cup s_k^{-1}\cdot A$ is (left) thick.
\end{enumerate}
Denote by $\syndetic \semistimes$, $\thick \semistimes$, and $\pws \semistimes$ the collections of all syndetic, thick, and piecewise syndetic subsets of the semigroup $\semistimes$. When the semigroup is apparent or unimportant, we refer to these classes simply as $\syndetic$, $\thick$, and $\pws$.
\end{definition}

We could define the classes of \emph{left syndetic}, \emph{right thick}, and \emph{left piecewise syndetic} sets but choose instead to relegate the analogous results for these ``opposite'' classes to a few interspersed remarks. The choice of \emph{left} and \emph{right} in Definition \ref{def:syndthick} makes thickness ``dual'' to syndeticity in the sense of the following definition.

\begin{definition}
Let $\arbclass \subseteq \subsets S$ be a collection of subsets of a set $S$. The \emph{dual class} $\arbclass^* \subseteq \subsets S$ is the collection of subsets of $S$ having non-empty intersection with every member of $\arbclass$; in other words, $A \in \arbclass^*$ if and only if for all $B \in \arbclass$, $A \cap B \neq \emptyset$.
\end{definition}

It is simple to check that $\syndetic^* = \thick$ and $\thick^* = \syndetic$. Note that if the collection $\arbclass$ is upward closed, then $(\arbclass^*)^* = \arbclass$. This construction allows us to define $\pws^*$, the dual class to the class of piecewise syndetic sets which appears in the statement of Theorem A.

\begin{lemma}\label{lem:basicfactsofduals}
Let $\semistimes$ be a semigroup
\begin{enumerate}[label=(\Roman*)]
\item \label{item:pwsisintersection} $A \in \pws$ if and only if there exist $C \in \syndetic$ and $T \in \thick$ such that $A = C \cap T$.
\item \label{item:combcharofpwsstar} $A \in \pws^*$ if and only if for all $F \in \finitesubsets S$, there exists $C \in \syndetic$ such that $F \tgen C \subseteq A$. ($F \tgen C$ denotes the set $\{ f \tgen c \ | \ f \in F, \ c \in C\}$.)
\item \label{item:psstarintersectedlarge} Let $A \in \pws$, $C \in \syndetic$, and $T \in \thick$.  If $P \in \pws^*$, then $A \cap P \in \pws$, $C \cap P \in \syndetic$, and $T \cap P \in \thick$.
\end{enumerate}
\end{lemma}

\begin{proof}
For a proof of \ref{item:pwsisintersection}, see \cite[Theorem 4.49]{hindmanstrauss-book}. To prove \ref{item:combcharofpwsstar}, note that $B \in \pws$ if and only if there exists $F \in \finitesubsets S$ and $T \in \thick$ such that for all $s \in T$, $Fs \cap B \neq \emptyset$. It follows that $A \in \pws^*$ if and only if for all $F \in \finitesubsets S$ and all $T \in \thick$, there exists $s \in T$ such that $Fs \subseteq A$. Now \ref{item:combcharofpwsstar} follows since $\thick^* = \syndetic$. A proof of the first assertion in \ref{item:psstarintersectedlarge} can be found in \cite[Lemma 9.4]{furstenberg-book}. The second and third assertions in \ref{item:psstarintersectedlarge} follow from similar set algebra and the characterization of piecewise syndeticity in \ref{item:pwsisintersection}.
\end{proof}

\begin{definition}
A semigroup $\semistimes$ is \emph{left cancellative} if for all $x,y,z \in S$, $x \tgen y = x \tgen z$ implies that $y=z$. \emph{Right cancellativity} is defined analogously. We say $\semistimes$ is \emph{cancellative} if it is both left and right cancellative.
\end{definition}

\begin{definition}\label{def:defofamenable}
A semigroup $\semistimes$ is \emph{left amenable} if the space of bounded, complex-valued functions on $S$ with the supremum norm admits a \emph{left translation invariant mean}, that is, a positive linear functional $\lambda$ of norm 1 which is left translation invariant: for all bounded $f: S \to \mathbb{C}$ and all $s \in S$, $\lambda \big( x \mapsto f(s\cdot x) \big) = \lambda ( f )$.
\end{definition}

The following definition and resulting characterization of upper Banach density appeared for $(\N,+)$ in \cite[Corollary 9.2]{Griesmer} and for general semigroups in \cite[Theorem G]{johnsonrichter}; see also \cite[Lemma 9.6]{furstenbergkatznelsonipszem}.

\begin{definition}\label{def:density}
Let $\semistimes$ be a semigroup and $A \subseteq S$. The \emph{(left) density of $A$} is
\[d_S^*(A) = \sup \big\{ \alpha \geq 0 \ \big| \ \forall F \in \finitesubsets S, \ \exists s \in S, \ \big|(F \tgen s) \cap A\big| \geq \alpha |F| \big\}.\]
When $\semistimes$ is cancellative and left amenable, this density coincides with the upper Banach density (see Lemma \ref{lem:basicfactsofdensity} \ref{item:dstarandbanachcoincide}) and we denote by $\density \semistimes$ the collection of all subsets of $\semistimes$ with positive density. (Throughout this paper, when referring to the class $\density \semistimes$, we shall implicitly assume that $\semistimes$ is cancellative and left amenable.)
\end{definition}

It is important to note that while the density function $\dstar_S$ may be ill-behaved in non-amenable semigroups (for example, the sets $aF_2$ and $bF_2$ have zero density yet form a partition of the free semigroup $F_2 = \langle a, b \rangle$), it still possesses some useful properties in arbitrary cancellative semigroups, as demonstrated in the following lemma.

\begin{lemma}\label{lem:basicfactsofdensity}
Let $\semistimes$ be a cancellative semigroup.
\begin{enumerate}[label=(\Roman*)]
\item \label{item:basicsofthick} $A \in \thick$ if and only if $d_S^*(A) = 1$.
\item\label{item:inversetranslatehasmoredensity} For all $x \in S$, $d_S^*(A) \leq d_S^*(x^{-1} \tgen A)$.
\end{enumerate}
Suppose $\semistimes$ is cancellative and left amenable.
\begin{enumerate}[label=(\Roman*)]
\setcounter{enumi}{2}
\item \label{item:dstarandbanachcoincide} The density $d_S^*$ coincides with the upper Banach density:
\[d_S^*(A) = \sup \big\{ \mu(1_A) \ \big| \ \mu \text{ a left translation invariant mean on } \semistimes \big\}.\]
\item \label{item:unionoftwozerodensity} If $A, B \subseteq S$ are such that $\dstar_S(A) = \dstar_S(B) = 0$, then $\dstar_S(A \cup B) = 0$.
\item \label{item:intersectionofdstar} If $D \in \density$ and $P \in \density^*$, then $D \cap P \in \density$.
\end{enumerate}
\end{lemma}

\begin{proof}
\ref{item:basicsofthick} is easy to verify (and does not require cancellativity). To prove \ref{item:inversetranslatehasmoredensity}, let $\al < d_S^*(A)$ and $F \in \finitesubsets S$. There exists an $x \in S$ such that $|s \cdot F \cdot x \cap A| \geq |s \cdot F| \al = |F| \al$. Note that left multiplication by $s$ is a bijection from the set $F\cdot x \cap s^{-1}\cdot A$ to the set $s\cdot F\cdot x \cap A$. Therefore, $|F\cdot x \cap s^{-1}\cdot A| \geq |F| \al$. Since $\al$ was arbitrary, the conclusion follows. \ref{item:dstarandbanachcoincide} follows from \cite[Theorem 3.5]{BGpaperone}, and \ref{item:unionoftwozerodensity} follows from the sub-additivity of the upper Banach density. Finally, \ref{item:intersectionofdstar} follows from \cite[Lemma 9.4]{furstenberg-book} and the fact that $\density$ is partition regular (a consequence of \ref{item:unionoftwozerodensity}).
\end{proof}

The most familiar appearance of the upper Banach density is perhaps for subsets $A \subseteq \N$, where it is given by
\begin{align*}\dstar_{\seminplus}(A) = \limsup_{n-m \to \infty} \frac{|A \cap \{m+1, \ldots, n\}|}{n-m}.\end{align*}
In cancellative, left amenable semigroups, the three definitions of density in Definition \ref{def:density}, in Lemma \ref{lem:basicfactsofdensity} \ref{item:dstarandbanachcoincide}, and as a supremum along all \folner nets all coincide; see \cite[Section 3]{BGpaperone}.

We describe now another notion of largeness, $\ip$ structure, which is of fundamental importance in Ramsey theory and ergodic theory; see \cite{hindmanoriginal}, \cite{fwdynamics}, \cite{furstenbergkatznelsonipszem}, and \cite{bergelsonleibmanannals}.

\begin{definition}\label{def:ipstructure}
Let $\semistimes$ be a semigroup and $A \subseteq S$.
\begin{enumerate}[label=(\Roman*)]
\item $A$ is an \emph{$\ipr$ set}, $r \in \N$, if it contains a set of the form
\begin{align*}\finiteproducts {(s_i)_{i=1}^r} = \big\{ s_{i_1} \tgen s_{i_{2}} \tgen \cdots \tgen s_{i_k} \ \big| \ 1 \leq k \leq r, \ \{i_1 < \cdots < i_k\} \subseteq \{1, \ldots, r\} \big\}\end{align*}
where $(s_i)_{i=1}^r$ is a sequence in $S$.
\item $A$ is an \emph{$\ipnaught$ set} if for all $r \in \N$, it is an $\ipr$ set.
\item $A$ is an \emph{$\ipset$ set} if it contains a set of the form
\begin{align*}\finiteproducts  {(s_i)_{i\in \N}} = \big\{ s_{i_1} \tgen s_{i_{2}} \tgen \cdots \tgen s_{i_k} \ \big| \ k \geq 1, \ \{i_1 < \cdots < i_k\} \subseteq \N \big\},\end{align*}
where $(s_i)_{i \in \N}$ is a sequence in $S$.
\end{enumerate}
We denote the class of $\ipr$, $\ipnaught$, and $\ip$ subsets of $\semistimes$ by $\iprclass \semistimes$, $\ipnaughtclass \semistimes$, and $\ipclass \semistimes$, respectively.
\end{definition}

In this definition, ``FP'' is short for ``finite products''; when the semigroup is written additively, we write ``FS,'' which abbreviates ``finite sums.'' The semigroup $\semistimes$ is not assumed to be commutative, so the order in which the products are taken is important. The increasing order was chosen here so that every (left) thick set is an $\ipset$ set (see Lemma \ref{lem:hierarchylemma} below); decreasing $\ipset$ sets (those defined with a decreasing order) can be found in any right thick set.

The following notion of largeness, centrality, combines the translation invariant notions first introduced and $\ip$ structure. While most of the main results in this paper do not concern the class of central sets directly, we mention it here for completeness. Central sets originated in $\N$ in a dynamical context with Furstenberg \cite[Definition 8.3]{furstenberg-book}. It was revealed in \cite[Section 6]{bhnonmetrizable} that the property of being central is equivalent to membership in a minimal idempotent ultrafilter.

\begin{definition}\label{def:defofcentral}
Let $\semistimes$ be a semigroup. A subset of $S$ is \emph{central} if it is a member of a minimal idempotent ultrafilter on $S$. We denote by $\central \semistimes$ the class of central subsets of $\semistimes$.
\end{definition}

The relationships between the classes of largeness presented thus far will be of critical importance in the following sections. A proof of the following lemma can be found in \cite[Section 1]{BHabundant}. That piecewise syndetic sets have positive density when $\semistimes$ is cancellative and left amenable is an immediate consequence of Lemma \ref{lem:basicfactsofdensity} and the definitions.

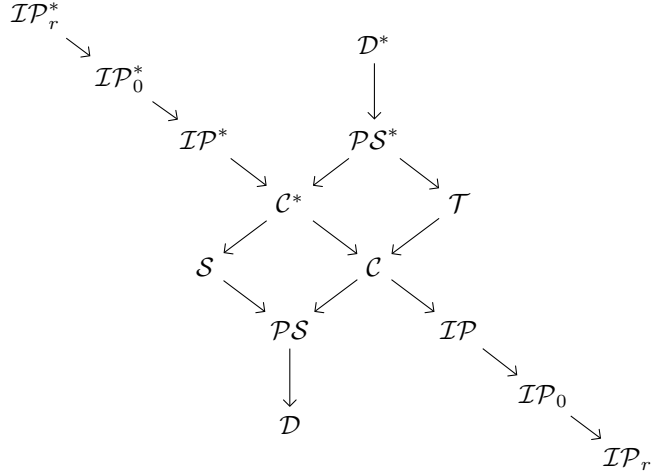
\begin{figure}[h]
\centering
\begin{tikzpicture}[>=triangle 60]
  \matrix[matrix of math nodes,column sep={16pt,between origins},row sep={12pt,between origins}](m)
  {
    |[name=iprstar]|\iprclass^* & & & & & & & & & & & & & & \\
    & & & & & & & & |[name=densitystar]|\density^* & & & & & & \\
    & & |[name=ipnaughtstar]|\ipnaughtclass^* \\
    \\
    & & & & |[name=ipstar]|\ipclass^* & & & & |[name=pwsstar]|\pws^* \\
    \\
    & & & & & & |[name=centralstar]|\central^* & & & & |[name=thick]| \thick \\
    \\
    & & & & |[name=syndetic]|\syndetic & & & & |[name=central]|\central \\
    \\
    & & & & & & |[name=pws]| \pws & & & & |[name=ip]|\ipclass\\
    \\
    & & & & & & & & & & & & |[name=ipnaught]|\ipnaughtclass \\
    & & & & & & |[name=density]|\density \\
    & & & & & & & & & & & & & & |[name=ipr]|\iprclass \\
 };
   \draw[-angle 90] (iprstar) edge (ipnaughtstar)
            (ipnaughtstar) edge (ipstar)
            (ipstar) edge (centralstar)
            (centralstar) edge (central)
            (central) edge (ip)
            (ip) edge (ipnaught)
            (ipnaught) edge (ipr)
            (densitystar) edge (pwsstar)
            (pwsstar) edge (thick)
            (pwsstar) edge (centralstar)
            (centralstar) edge (syndetic)
            (syndetic) edge (pws)
            (thick) edge (central)
            (central) edge (pws)
            (pws) edge (density)
  ;
\end{tikzpicture}
\caption{Containment amongst classes of largeness in a semigroup.}
\label{fig:containmentdiagram}
\end{figure}

\begin{lemma}\label{lem:hierarchylemma}
Let $\semistimes$ be a semigroup and $r \in \N$. The diagram in Figure \ref{fig:containmentdiagram} illustrates containment amongst the classes of largeness in $\semistimes$, with $\mathcal{X} \to \mathcal{Y}$ indicating that $\mathcal{X} \subseteq \mathcal{Y}$. The classes $\density$ and $\density^*$ are only considered in the case that $\semistimes$ is cancellative and left amenable.
\end{lemma}

\section{\texorpdfstring{$\Q$}{Q}-algebras and alignment}\label{sec:ringreps}  

Fix $d \in \N$, and denote by $\discrims_d$ the set of all binary operations $\tone: \Zd \times \Zd \to \Zd$ which make the abelian group $(\Zd,+)$ into a (not-necessarily commutative or unital) ring $(\Zd,+,\tone)$ without zero divisors. We call the elements of $\discrims_d$ \emph{proper multiplications on $\Zd$}, omitting the subscript $d$ and the word ``proper'' when it is otherwise clear. All of the countably many multiplications $\tone \in \discrims$ make $(\Zdz,\tone)$ a cancellative semigroup. For each of the classes $\arbclass \in \{\syndetic,\thick,\pws,\ldots\}$ introduced in Section \ref{sec:defs}, we abbreviate $\arbclass(\Zdz,\tone)$ by $\arbclass(\tone)$.

The ring $(\Zd,+,\tone)$, resp. semigroup $(\Zdz,\tone)$, is best understood in the context of the $\Q$-algebra, resp. group, into which it embeds. The multiplication $\tone$ extends uniquely by linearity to a binary operation $\widetilde{\tone}: \Qd \times \Qd \to \Qd$ which makes the vector space $\Qd$ into an associative algebra $(\Qd,+,\widetilde{\tone})$ over $\Q$. A finite dimensional associative algebra with no zero divisors is a division algebra\footnote{Let $(\mathcal{A},+,\cdot)$ be such an algebra, and let $a \in \mathcal{A}$ be non-zero. Since $\mathcal{A}$ has no zero divisors, multiplication on the left and right by $a$ is injective, hence surjective, and so there exists $u \in \Qd$ for which $a \cdot u = a$. Note that $a \cdot u \cdot a = a \cdot a$, whereby $u \cdot a = a$. Given $b \in \mathcal{A}$, write $b = c \cdot a$ to see that $b \cdot u = b$. This shows that $u$ is the multiplicative identity. It follows from surjectivity again that there exists $d \in \mathcal{A}$ for which $a \cdot d = u$. As before, $d \cdot a = u$, so $d$ is the inverse of $a$.} (a unital algebra in which every non-zero element has a two-sided inverse). The cancellative semigroup $(\Zdz,\tone)$ is thus embedded in the group $(\Qdz,\widetilde{\tone})$.

Denote by $\qdiscrims_d$ the collection of all binary operations $\widetilde{\tone}: \Q^d \times \Q^d \to \Q^d$ which make $(\Qd,+,\widetilde{\tone})$ into an associative division algebra over $\Q$. We will consider $\discrims_d$ as a subset of $\qdiscrims_d$ by omitting the tilde notation and automatically extending multiplications on $\Zd$ to ones on $\Qd$.

Representations of these division algebras are an important tool in this paper. For $R$ a commutative ring with identity, let $\matdR$ be the ring of $d$-by-$d$ matrices over $R$ and $\gldR$ be the group of $d$-by-$d$ matrices whose determinant is a unit in $R$. Abusing notation, we will regard $T \in \matdR$ as both a matrix and as an $R$-linear function from $R^d$ to $R^d$. Let $e_i = (0,\ldots,1_R,\ldots,0) \in R^d$ be the $i^{\text{th}}$ standard unit vector, and let $\id$ be the identity matrix.

\begin{lemma}\label{lem:repproperties}
Let $\tone \in \qdiscrims$.  The left (right) representation
\begin{align*}
\rone: (\Qd,+,\tone) &\to (\matdq,+,\cdot) & \roneright: (\Qd,+,\tone) &\to (\matdq,+,\cdot)\\
x &\mapsto (x \tone e_1 \cdots x \tone e_d) & x &\mapsto (e_1 \tone x \cdots e_d \tone x)
\end{align*}
is an injective $\Q$-algebra (anti-)homomorphism satisfying: for all $x, y \in \Qd$,
\[x \tone y = \rone(x)y = \roneright(y)x.\]
Moreover, $\rone: (\Qdz,\tone) \to (\gldq,\cdot)$ is a group homomorphism. If $\tone \in \discrims$, then
\begin{align*} \rone: (\Zd,+,\tone) \to (\matdz,+,\cdot) \text{ and } \rone: (\Zdz,\tone) \to (\matdz \cap \gldq,\cdot) \end{align*}
are ring and semigroup homomorphisms, respectively. The same statements apply with $\rone$ replaced by $\roneright$ and ``anti-'' prepended to ``homomorphism.''
\end{lemma}

The proof of this lemma is standard and is omitted. We will usually refer to the left representation of $(\Qd,+,\tone)$ as simply ``the representation of $\tone$.''  Throughout, we will consistently denote the representations of the multiplications $\tone$ and $\ttwo$ by $\rone$ and $\rtwo$, respectively.

\begin{lemma}\label{lem:reachestheid}
Let $\tone \in \qdiscrims$. For all $x \in \Qdz$, there exists $b \in \N$ such that $b \rone(x)^{-1} \in \ronezdz$. In particular, there exists $c \in \N$ and $w \in \Zdz$ for which $\rone(w) = \roneright(w) = c\idd$.
\end{lemma}

\begin{proof}
Since $(\Qd,+,\tone)$ is a division algebra, there exists a $z \in \Qdz$ such that $\rone(x)^{-1} = \rone(z)$. Now $b \in \N$ can be taken to be the least common multiple of the denominators of the coordinates of $z$.

If $c \in \N$ is such that $w = c (z \tone x) \in \Zdz$, then $\rone(w) = c \rone(z) \rone(x) = c \idd$. To see that $\roneright(w) = c\idd$, note that
\[\roneright(e_1) \roneright(w) = \roneright \big( \rone(w) e_1 \big) = \roneright(c e_1) = c\roneright(e_1).\]
Since $\roneright(e_1)$ is invertible, we see $\roneright(w) = c\idd$.
\end{proof}

We now explain the condition (\ref{eqn:mainconditiontheorema}) appearing in Theorem A in the introduction in terms of the images of the representations of the algebras involved. For $U \subseteq \matdq$, the \emph{centralizer of $U$ in $\matdq$} is the set of matrices in $\matdq$ which commute with all matrices in $U$. The following lemma connects the centralizer of the image of the left representation of $\tone$ with the image of its right representation (cf. \cite[Lemma 2.45]{knappbook}).

\begin{lemma}\label{lem:centralizerdescription}
Let $\tone \in \qdiscrims$. The centralizer of $\roneqd$ in $\matdq$ is $\ronerightqd$.
\end{lemma}

\begin{proof}
$\matdq$ is a $d^2$-dimensional, central simple algebra over $\Q$, and $\roneqd$ is a $d$-dimensional, simple subalgebra of $\matdq$. By the Double Centralizer Theorem \cite[Theorem 2.43]{knappbook}, the centralizer of $\roneqd$ is a $d$-dimensional, simple subalgebra of $\matdq$. By associativity, multiplication on the left commutes with multiplication on the right, so the centralizer contains $\ronerightqd$, a $d$-dimensional subalgebra of $\matdq$. Because the dimensions coincide, the lemma follows.
\end{proof}

The first of the equivalent conditions in the following lemma is condition (\ref{eqn:mainconditiontheorema}) in Theorem A from the introduction. By a \emph{$d$-dimensional lattice in $\matdq$}, we mean a free subgroup of $(\matdq,+)$ with $d$ generators.

\begin{lemma}\label{lem:equivconditionsforaligned}
Let $\tone, \ttwo \in \qdiscrims$, and let $\rone, \rtwo$ be their representations. The following are equivalent:
\begin{enumerate}[label=(\Roman*)]
\item\label{item:ltone} there exist $v, w \in \Zdz$ such that for all $x,y \in \Zd$,
\[x \tone v \tone y = x \ttwo w \ttwo y;\]
\item\label{item:lttwo} $\ronezd \cap \rtwozd$ is a $d$-dimensional lattice in $\matdq$;
\item\label{item:ltthree} $\roneqd = \rtwoqd$;
\item\label{item:ltfour} there exists $v \in \Qd$ for which $\rtwo = \rone \circ \roneright(v)$ (where we abuse notation by regarding the matrix $\roneright(v)$ as a function $\Q^d \to \Q^d$).
\end{enumerate}
\end{lemma}

\begin{proof}
Suppose \ref{item:ltone} holds, and choose $c \in \N$ such that $c\roneright(v), c\rtworight(w) \in \matdz$. Writing \ref{item:ltone} in terms of representations, for all $x \in \Zd$, $\rone(c\roneright(v)x) = \rtwo(c\rtworight(w)x)$. Since $\roneright(v)$ and $\rtworight(w)$ are invertible, we have that $\rone(c\roneright(v)\Zd) = \rtwo(c\rtworight(w)\Zd)$ is a $d$-dimensional lattice contained in $\ronezd \cap \rtwozd$, a lattice of dimension at most $d$.

To see that \ref{item:lttwo} implies \ref{item:ltthree}, suppose $\ronezd \cap \rtwozd = \zspan ( \{ \ell_1, \ldots, \ell_d \} )$. It follows that both $\roneqd$ and $\rtwoqd$ contain $\qspan ( \{ \ell_1, \ldots, \ell_d \} )$, a $d$-dimensional linear subspace of $\matdq$. Since $\roneqd$ and $\rtwoqd$ are $d$-dimensional linear subspaces of $\matdq$, they coincide.

To see that \ref{item:ltthree} implies \ref{item:ltfour}, suppose $\roneqd = \rtwoqd = \qspan ( \{ \ell_1, \ldots, \ell_d \} )$.  Let $T$ be the transformation taking $\rtwo^{-1}(\ell_i)$ to $\rone^{-1}(\ell_i)$, and note that $T \in \gldq$ and $\rtwo = \rone \circ T$. Since $\rone$ and $\rtwo$ are ring homomorphisms, for all $x, y \in \Qd$,
\[\rone( T \rone( Tx) y) = \rtwo(\rtwo(x)y)= \rtwo(x)\rtwo(y) = \rone(Tx) \rone(Ty) = \rone( \rone(Tx) Ty).\]
Since $\rone$ is injective and the previous equality holds for all $y \in \Qd$, we conclude that for all $x \in \Qd$,
\[T \rone( Tx) = \rone(Tx) T.\]
Since $T \in \gldq$ and the previous equality holds for all $x \in \Qd$, we conclude that $T$ is in the centralizer of $\roneqd$ in $\matdq$. By Lemma \ref{lem:centralizerdescription}, there exists $v \in \Qd$ for which $T = \roneright(v)$. Therefore, \ref{item:ltthree} implies \ref{item:ltfour}.

Assuming \ref{item:ltfour}, for all $x \in \Qd$, $\rtwo (x) = \rone (\roneright(v)x)$. By Lemma \ref{lem:reachestheid}, there exists $w \in \Qdz$ such that $\rtworight(w) = \idd$. Choose $c \in \N$ so that $cv, cw \in \Zdz$. It follows that for all $x, y \in \Qd$, 
\[x \tone (cv) \tone y = \rone(\roneright(cv)x)y = c\rtwo(x)y = \rtwo(\rtworight(cw)x)y = x \ttwo (cw) \ttwo y.\]
This yields \ref{item:ltone}, finishing the proof.
\end{proof}

\begin{definition}\label{def:aligned}
Two multiplications $\tone, \ttwo \in \qdiscrims$ are
\begin{enumerate}[label=(\Roman*)]
\item \label{item:aligneddef}\emph{aligned} if any (all) of the equivalent conditions in Lemma \ref{lem:equivconditionsforaligned} hold; 
\item \emph{isomorphic} if $(\Qd,+,\tone)$ and $(\Qd,+,\ttwo)$ are isomorphic $\Q$-algebras;
\item \emph{opposite} if $\ttwo = \toneop$, where $\toneop$ is defined by $x \toneop y = y \tone x$.
\end{enumerate}
\end{definition}

The word ``aligned'' reflects the fact that the images of the representations coincide. The previous lemma makes it clear that ``aligned'' is an equivalence relation on $\qdiscrims$ and hence on $\discrims$. Note that $\tone, \ttwo \in \discrims$ being isomorphic does \emph{not} mean that the rings $(\Zd,+,\tone)$, $(\Zd,+,\ttwo)$ or the semigroups $(\Zdz,\tone)$, $(\Zdz,\ttwo)$ are isomorphic. For example, though $\Z \big[(1+\sqrt 5)/2 \big]$ and $\Z \big[\sqrt 5 \big]$ are non-isomorphic rings, the multiplications they induce on $\Z^2$ are isomorphic according to Definition \ref{def:aligned} because the enveloping $\Q$-algebras are both isomorphic to $\Q(\sqrt 5)$.

The equivalence relations ``aligned,'' ``isomorphic,'' and ``opposite'' are intimately related with certain subspaces of $\gldq$. This relationship will be important in the proof of Corollary C in Section \ref{sec:proofofcorollary}. For $\tone, \ttwo \in \qdiscrims$, let
\begin{align*}
\centshort(\tone) &= \{T \in \gldq \ | \ \forall x \in \Qdz, \ \rone(x) T = T\rone(x) \},\\
\normshort(\tone) &= \{T \in \gldq \ | \ T^{-1} \roneqdz T = \roneqdz \},\\
\isom(\tone,\ttwo) &= \{T \in \gldq \ | \ T:(\Qd,+,\tone) \to (\Qd,+,\ttwo) \text{ isomorphism} \},\\
\aut(\tone) &= \isom(\tone,\tone),
\end{align*}
where ``$\centshort$,'' ``$\normshort$,'' and ``$\isom$,'' are short for ``centralizer,'' ``normalizer,'' and ``unital $\Q$-algebra isomorphism.'' The following $\gldq$-action on $\qdiscrims$ serves to connect the relations in Definition \ref{def:aligned} with these spaces. Note that if $\tone$ and $\toneop$ are not isomorphic, then $\isom(\tone,\toneop)$ is empty.

\begin{lemma}\label{lem:gldqactionproperties}
The map
\[\begin{aligned} \gldq \times \qdiscrims &\to \qdiscrims \\ (T,\tone) &\mapsto \tone_T \end{aligned} \qquad \text{defined by} \qquad x \tone_T y := T^{-1} \big( T x \tone T y \big)\]
yields a right $\gldq$-action on $\qdiscrims$. Moreover, for $\tone, \ttwo \in \qdiscrims$,
\[\tone, \ \ttwo \text{ are  } \begin{cases} \text{isomorphic} \\ \text{aligned} \\ \text{equal} \\ \text{opposite} \end{cases} \! \! \! \! \! \! \text{  iff there exists } T \in \begin{cases} \gldq \\ \normshort(\tone) \\ \aut(\tone) \\ \isom(\tone,\toneop) \end{cases} \! \! \! \! \! \! \text{ s.t. } \ttwo = \tone_T.\]
\end{lemma}

\begin{proof}
It is routine to check that this indeed defines a right $\gldq$-action on the set $\qdiscrims$. Let $T \in \gldq$ and $\tone \in \qdiscrims$. One can easily check that
\[T: (\Qd,+,\tone_T) \longrightarrow (\Qd,+,\tone)\]
is a $\Q$-algebra isomorphism, meaning $\tone$ and $\tone_T$ are isomorphic. Conversely, suppose $\ttwo \in \qdiscrims$ is isomorphic to $\tone$ via $T \in \isom(\ttwo,\tone)$. Since $T(x \ttwo y) = T(x) \tone T(y)$, $\ttwo = \tone_{T}$.

Note that for any $T \in \gldq$, the representation $\rone_{T}$ of $\tone_T$ is defined on $\Qd$ by
\begin{align}\label{eqn:repofactedonmult}\rone_{T} (x) = T^{-1}\rone(Tx) T.\end{align}
If $\tone$ and $\ttwo$ are aligned, by the proof of Lemma \ref{lem:equivconditionsforaligned}, there exists $T$ in the centralizer of $\roneqdz$ in $\gldq$ (and so, in particular, in $\normshort(\tone)$) such that $\rtwo = \rone \circ T$. Since $T$ commutes with matrices in $\roneqdz$, it follows from (\ref{eqn:repofactedonmult}) that $\rone_T = \rone \circ T = \rtwo$, meaning $\ttwo = \tone_T$. Conversely, if $T \in \normshort(\tone)$, then we see from (\ref{eqn:repofactedonmult}) that $\rone_{T}(\Qd) = \rone(\Qd)$, meaning $\tone_T$ and $\tone$ are aligned.

It follows from the first paragraph of this proof that if $\tone_T = \tone$, then $T$ is an automorphism of $(\Qd,+,\tone)$.

Finally, note that $\cdot_{\text{op}}: \qdiscrims \to \qdiscrims$ commutes with the $\gldq$-action on $\qdiscrims$. If $\tone$ and $\tone_T$ are opposite,  then $(\toneop)_T = \tone$, and the first paragraph of this proof gives that $T$ is an isomorphism from $(\Qd,+,\tone)$ to $(\Qd,+,\toneop)$.
\end{proof}

We give now a brief description of the subspaces of $\gldq$ appearing here; concrete examples are given in the following section. It follows from Lemma \ref{lem:centralizerdescription} that $\centshort(\tone) = \roneright(\Qdz)$, and by standard group theory, this is a normal subgroup of $\normshort(\tone)$. Moreover,
\begin{align}\label{eqn:normisautcent}\normshort(\tone) = \aut(\tone) \centshort(\tone).\end{align}
To see this, let $T \in \normshort(\tone)$. Since $\tone$ and $\tone_T$ are aligned, we see from the proof of Lemma \ref{lem:gldqactionproperties} that there exists $S \in \centshort(\tone)$ such that $\tone_T = \tone_S$. It follows from the lemma that $TS^{-1} \in \aut(\tone)$, meaning $\normshort(\tone) \subseteq \aut(\tone) \centshort(\tone)$. The reverse inclusion is simple to verify, yielding (\ref{eqn:normisautcent}).

Suppose $\isom(\tone,\toneop)$ is non-empty, and let $T \in \isom(\tone,\toneop)$. Since $T \isom(\tone,\toneop) \subseteq \aut(\tone)$ and $T^2 \in \aut(\tone)$, $\isom(\tone,\toneop) \subseteq T \aut(\tone)$. The reverse inclusion is immediate, meaning that $\isom(\tone,\toneop) = T \aut(\tone)$.

We end this section with a remark on the amenability of the semigroups $(\Zdz,\tone)$.

\begin{lemma}\label{lem:nonamenable}
Let $\tone \in \discrims$. The semigroup $(\Zdz,\tone)$ is left amenable if and only if $(\Zdz,\tone)$ is commutative.
\end{lemma}

\begin{proof}
It is a classical fact that a semigroup which is commutative is left amenable; see, for example, \cite[Theorem 4]{awstrongfolner}.

To handle the other direction, note that by Lemma \ref{lem:reachestheid}, the group $G = (\Qdz,\tone)$ is a group of right (and left) quotients of the semigroup $S=(\Zdz,\tone)$; this means that when $S$ is identified as a subset of $G$ in the obvious way, $G = S \tone S^{-1} = S^{-1}\tone S$. It follows by \cite[Corollary 2]{grigorstepin} that the semigroup $S$ is left amenable if and only if the group $G$ is amenable.

Suppose $(\Zdz,\tone)$ is not commutative. Because $G$ is the multiplicative group of a finite dimensional, non-commutative division algebra, it follows from \cite[Lemma 2.0]{goncalves} that $G$ contains a non-cyclic free subgroup. It is classical fact that a group containing a non-cyclic free subgroup is non-amenable. Hence, $S$ is not left (or right) amenable.
\end{proof}

\section{Examples of proper multiplications on \texorpdfstring{$\Zd$}{Zd}}\label{sec:earlyexamples} 

In this section, we describe the sets of proper multiplications on $\Z$ and $\Zt$ and discuss an example of a non-commutative proper multiplication on $\Z^4$. We also give concrete descriptions of the subspaces of $\gldq$ appearing in the previous section.

For non-zero $n \in \Z$, let $\tthree n$ be the multiplication on $\Z$ defined by $x \tthree n y = nxy$. Thus, $\tthree 1$ is the usual multiplication on $\Z$. It is an easy exercise to show that 
\[\discrims_1 = \big\{ \tthree n \ \big| \ n \in \Zz\big\}.\]
Moreover, for all non-zero $n, m, x, y \in \Z$,
\[x \tthree n m^2 \tthree n y = x \tthree m n^2 \tthree m y = (nm)^2 x y,\]
so each pair of multiplications in $\discrims_1$ is aligned. It follows by Theorem A (proven in the next section) that there is only a single notion of ``multiplicatively \psstar{}'' in $\Zz$. Each multiplication in $\discrims_1$ extends to a multiplication in $\qdiscrims_1$, and it is easy to see that
\[\qdiscrims_1 = \big\{ \tthree r \ \big| \ r \in \Qz\big\},\]
where $x \tthree r y = rxy$. All multiplications in $\qdiscrims_1$ are aligned, hence isomorphic.

A natural way to create proper multiplications on $\Zd$ when $d \geq 2$ is to consider the ring $\Z[x] \big / \big(p(x) \big)$ when $p(x) \in \Z[x]$ is a monic, irreducible polynomial of degree $d$. Denote by $\tthree {p(x)}$ the multiplication on $\Zd$ gotten by associating the class of $\sum_{i=0}^{d-1} a_i x^i$ in $\Z[x] \big / \big(p(x) \big)$ with the integer vector $(a_i)_{i=0}^{d-1} \in \Zd$. Note that $(\Zd,+,\tthree {p(x)})$ is a ring without zero divisors and with multiplicative identity $e_1$.

Let $\rthree {p(x)}$ be the representation of $\tthree {p(x)} \in \qdiscrims_d$. Since $\tthree {p(x)}$ is commutative, $\centshort(\tthree {p(x)}) = \rthree {p(x)}(\Qdz)$, $(\tthree {p(x)})_{\text{op}}=\tthree {p(x)}$, $(\Qd,+,\tthree {p(x)})$ is a degree $d$ field extension over $\Q$, and $\aut(\tthree {p(x)})$ is its Galois group over $\Q$. Since the Galois group of a finite extension of $\Q$ is finite, $\aut(\tthree {p(x)})$ is finite and, by (\ref{eqn:normisautcent}), $\normshort(\tthree {p(x)})$ is a finite union of cosets of $\rthree {p(x)}(\Qdz)$.

This construction allows us to describe all proper multiplications on $\Zt$. Indeed, by the classification of quadratic rings,
\begin{align}\label{eqn:setsofquadringsmults}\big\{ \tthree {x^2-bx - c} \ \big | \ b \in \{0,1\}, \ c \in \Z \setminus \{0^2, 1^2, 2^2, \ldots\} \big\} \subseteq \discrims_2,\end{align}
is a complete set (up to \emph{unital ring} isomorphism) of proper multiplications on $\Zt$ with a multiplicative identity; see \cite[Section 3]{bhargava} and the subsequent papers in the series for parameterizations of quadratic, cubic, quartic, and quintic rings. The remaining multiplications in $\qdiscrims_2$ can be ``reached'' by acting on these via the $\text{GL}_2(\Q)$-action described in Section \ref{sec:ringreps}.

When $p(x) = x^2-bx-c$, the representation $\rthree {p(x)}$ of $\tthree {p(x)}$ can be calculated by hand as
\[\rthree {p(x)} (x) = \matrixx {x_1}{c x_2}{x_2}{x_1 + b x_2}.\]
It is easy to see that the subspace $\rthree {p(x)} (\Q^2)$ is uniquely determined by the values of $b$ and $c$, so Lemma \ref{lem:equivconditionsforaligned} gives that no pair of multiplications in the set in (\ref{eqn:setsofquadringsmults}) are aligned. It follows from Theorem A and Lemma \ref{lem:equivconditionsforaligned} that ``multiplicatively \psstar{}'' in $\Ztz$ is a distinct notion for each of these multiplications.

When $p(x) = x^2-c$, the automorphism group $\aut(\tthree {p(x)})$ is
\[\aut(\tthree {p(x)}) = \left\{ \matrixx 1001, \matrixx 100{-1} \right\},\]
where the non-trivial automorphism corresponds to the usual involution $\sqrt{c} \mapsto -\sqrt{c}$ of $\Z \big[\sqrt{c} \big]$.
By (\ref{eqn:normisautcent}), we can write the normalizer explicitly as
\[\normshort(\tthree {p(x)}) = \left\langle \matrixx 1001, \matrixx {0}{c}{1}{0} \right\rangle_\Q \mcup \left\langle \matrixx 100{-1}, \matrixx {0}{c}{-1}{0} \right\rangle_\Q.\]
Corollary C (as formulated and proved in Section \ref{sec:proofofcorollary}) gives that a matrix $T \in \matdz$ with non-zero determinant preserves $\pws^*(\tthree {p(x)})$ if and only if $T \in \normshort(\tthree {p(x)})$.

The previous construction generates a plethora of commutative multiplications in $\discrims_3$, $\discrims_4$, and beyond. Because the dimension of a finite dimensional division algebra over its center is a square (\cite[Corollary 2.40]{knappbook}), the center of every three dimensional division algebra is the entire division algebra. It follows that every multiplication in $\qdiscrims_3$, and hence in $\discrims_3$, is commutative.

The Lipschitz integral quaternions
\[\quatz = \big\{ x_1 + ix_2 + jx_3 + kx_4 \ \big| \ x_1, x_2, x_3, x_4 \in \Z \big\} \]
provide an example of a non-commutative proper multiplication $\tone_\quat$ on $\Z^4$. The ring $(\Z^4,+,\tone_\quat)$ is a subring of $(\Q^4,+,\tone_\quat)$, a quaternion division algebra over $\Q$. In fact, all non-commutative multiplications in $\qdiscrims_4$ arise from so-called \emph{generalized quaternion algebras} (\cite[Chapter II, 11.16]{knappbook}).

We can compute the left and right representations of $\tone_\quat$ explicitly as
\[\rone_\quat(x) = \left( \begin{array}{cccc} x_1 & -x_2 & -x_3 & -x_4 \\ x_2 & x_1 & -x_4 & x_3 \\ x_3 & x_4 & x_1 & -x_2 \\ x_4 & -x_3 & x_2 & x_1 \end{array} \right), \ \rone_{\quat,\text{r}}(x) = \left( \begin{array}{cccc} x_1 & -x_2 & -x_3 & -x_4 \\ x_2 & x_1 & x_4 & -x_3 \\ x_3 & -x_4 & x_1 & x_2 \\ x_4 & x_3 & -x_2 & x_1 \end{array} \right).\]
Because $(\Q^4,+,\tone_\quat)$ is a central simple $\Q$-algebra, all of its automorphisms are inner (\cite[Corollary 2.42]{knappbook}). Combined with the work from the previous section, it follows that
\begin{align*}
\aut(\tone_\quat) &= \big\{ \rone_{\quat}(x) \rone_{\quat,\text{r}}(x)^{-1} \ \big | \ x \in \Q^4_{\neq 0}\big\} \\
\centshort(\tone_\quat) &=\rone_{\quat,\text{r}}(\Q^4_{\neq 0}), \\
\normshort(\tone_\quat) &= \rone_{\quat}(\Q^4_{\neq 0}) \rone_{\quat,\text{r}}(\Q^4_{\neq 0}).
\end{align*}
In this case, $\tone_{\quat}$ and $(\tone_{\quat})_{\text{op}}$ are isomorphic via $i, j, k \mapsto -i, -j, -k$, respectively. By the work in the previous section,
\[\isom \big(\tone_{\quat},(\tone_{\quat})_{\text{op}} \big) = \left( \begin{array}{cccc} 1&0&0&0 \\ 0&-1&0&0 \\ 0&0&-1&0 \\ 0&0&0&-1 \end{array} \right) \aut(\tone_\quat).\]

The sets $\aut(\tone) \cap \matdz$ and $\isom (\tone,\toneop ) \cap \matdz$ will appear in Section \ref{sec:proofofcorollary}, so we make particular mention of them here. It can be shown that $\aut(\tone_\quat) \cap M_4(\Z)$ is a finite subgroup of $\aut(\tone_\quat)$ isomorphic to $S_4$, the symmetric group on four elements; these correspond to the automorphisms of the Lipschitz quaternions. It follows that $\isom \big(\tone_{\quat},(\tone_{\quat})_{\text{op}} \big) \cap M_4(\Z)$ also consists of 24 matrices.


\section{Proof of Theorem A}\label{sec:manymults}  

We will prove the ``if'' and ``only if'' statements in Theorem A separately as Theorems \ref{thm:equaldensityfunctions} and \ref{thm:thickbutwithoutdensity} below. Recall the density function $\dstar_\tone$ defined in Definition \ref{def:density}.

\begin{theorem}\label{thm:equaldensityfunctions}
If $\tone, \ttwo \in \discrims$ are aligned, then $\dstar_\tone = \dstar_\ttwo$.
\end{theorem}

\begin{proof}
Let $v, w \in \Zdz$ be as in Definition \ref{def:aligned}. Let $A \subseteq \Zdz$ and $\alpha < \dstar_\tone(A)$. Let $F \in \finitesubsets \Zdz$, and note that by cancellativity, $F \tone v \in \finitesubsets \Zdz$ has the same cardinality as $F$. Since $\alpha < \dstar_\tone(A)$, there exists $x \in \Zdz$ such that
\[\alpha \leq \frac{\big| (F \tone v \tone x) \cap A\big|}{|F \tone v|} = \frac{\big| (F \ttwo w \ttwo x) \cap A\big|}{|F|}.\]
Since $F$ was arbitrary, $\dstar_\ttwo(A) \geq \alpha$. Since $\alpha < \dstar_\tone(A)$ was arbitrary, $\dstar_\ttwo(A) \geq \dstar_\tone(A)$. Since $A$ was arbitrary, $\dstar_\ttwo \geq \dstar_\tone$. The conclusion now follows by symmetry.
\end{proof}

\begin{remark}\label{rem:folnersequences}
Suppose $\tone$ and $\ttwo$ are aligned. Using the same idea in the proof of Theorem \ref{thm:equaldensityfunctions}, it can be shown that $(F_n)_{n\in \N} \subseteq \finitesubsets \Zdz$ is a \emph{\folner sequence} with respect to $\tone$ (for all $x \in \Zdz$, $\big| (x \tone F_n) \bigtriangleup F_n \big| \big/ |F_n| \to 0$ as $n \to \infty$) if and only if it is a \folner sequence with respect to $\ttwo$. 
\end{remark}

\begin{corollary}\label{cor:ifdirectionofmaintheoremone}
Let $\arbclass \in \{\syndetic,\thick,\pws,\pws^*,\density,\density^*\}$. If $\tone, \ttwo \in \discrims$ are aligned, then $\arbclass(\tone) = \arbclass(\ttwo)$.
\end{corollary}

\begin{proof}
When $(\Zdz,\tone)$ and $(\Zdz,\ttwo)$ are left amenable, Theorem \ref{thm:equaldensityfunctions} gives the corollary for $\density$, hence also for $\density^*$. In general, by Lemma \ref{lem:basicfactsofdensity} \ref{item:basicsofthick}, the corollary holds for $\thick$, hence also for $\thick^* = \syndetic$. Therefore, by Lemma \ref{lem:basicfactsofduals} \ref{item:pwsisintersection}, the corollary holds for $\pws$, hence also for $\pws^*$.
\end{proof}

\begin{remark}
By the symmetry in Definition \ref{def:aligned} \ref{item:aligneddef}, it is easy to see that $\tone$, $\ttwo$ are aligned if and only if $\toneop, \ttwoop$ are aligned. This means that the previous corollary holds for the ``opposite'' classes as well: left syndetic, right thick, left piecewise syndetic, etc...
\end{remark}

To prove the other half of the Theorem A, we assume that $\tone$ and $\ttwo$ are not aligned and construct a set which is ``large'' with respect to $\tone$ but ``small'' with respect to $\ttwo$. In fact, we prove much more in Theorem \ref{thm:thickbutwithoutdensity} below: given any collection of multiplications $\bigmults \subseteq \discrims$, there exists a set which is thick with respect to every multiplication in $\bigmults$ but has zero density with respect to every multiplication not aligned with any multiplication in $\bigmults$.

To construct such a set, we take the union of $\tone$-dilations of the ``cubes''
\[\cube N = \{-N, \ldots, N\}^d \setminus \{0\} \subseteq \Zdz.\]
In order for this set to have zero density with respect to $\ttwo$, the constituent subsets $\cube N \tone x$ must be carefully chosen so as to ``avoid'' $\ttwo$-dilations of finite sets. This idea is captured in the following definition. A subset of $\Zd$ is \emph{linearly dependent over $\Z$} if there exists a non-zero finite $\Z$-linear combination of the elements equal to zero.

\begin{definition}\label{def:avoiding}
Let $\ttwo \in \discrims$ and $G \in \finitesubsets \Zdz$. A set $A \subseteq \Zdz$ is \emph{$(G,\ttwo)$-avoiding} if for all for all $z \in \Zdz$ and for all $F \subseteq G$ with the property that no subset of $d$ points of $F$ is linearly dependent over $\Z$,
\[\big|(F \ttwo z ) \cap A\big| \leq d-1.\]
For $\finitediscrims \subseteq \discrims$, the set $A$ is \emph{$\left(G, \finitediscrims \right)$-avoiding} if for all $\ttwo \in \finitediscrims$, it is $(G,\ttwo)$-avoiding.
\end{definition}

\begin{lemma}\label{lem:thereareavoidingsets}
Let $\tone \in \discrims$, and let $\finitediscrims \subseteq \discrims$ be a finite collection of multiplications, none of which is aligned with $\tone$. For all $G \in \finitesubsets \Zdz$, there exist infinitely many $x \in \Zdz$ for which $G \tone x$ is $\left(G, \finitediscrims \right)$-avoiding.
\end{lemma}

\begin{proof}
Let $G \in \finitesubsets \Zdz$, and let $\rone$ be the representation of $\tone$. For $f, g \in G$ and $\ttwo \in \finitediscrims$, let $T_{\ttwo,f,g} = \rtwo(f)^{-1} \rone(g)$, where $\rtwo$ is the representation of $\ttwo$. Note that $T_{\ttwo,f,g} \in \gldq$. We will show that for all $x$ in the set
\begin{align}\label{eqn:hugegoodset}\Zdz \setminus \bigcup_{\substack{\ttwo \in \finitediscrims \\ f,f',g,g' \in G \\ T_{\ttwo,f,g} \neq T_{\ttwo,f',g'}}} \nullspace \big(T_{\ttwo,f,g} - T_{\ttwo,f',g'} \big),\end{align}
the set $G \tone x$ is $\left(G, \finitediscrims \right)$-avoiding. This suffices to prove the lemma since the set in (\ref{eqn:hugegoodset}) is infinite; indeed, it is $\Zdz$ with a finite number of strictly lower-dimensional linear subspaces removed.

Let $x$ be an element of the set in (\ref{eqn:hugegoodset}). Fix $\ttwo \in \finitediscrims$, let $\rtwo$ be its representation, and write
\begin{align}\label{eqn:expressionfortrans} T_{f,g} = \rtwo(f)^{-1} \rone(g).\end{align}
To show that $G \tone x$ is $(G, \ttwo)$-avoiding, it suffices to prove: for all $F \subseteq G$ with $|F|=d$, if there exists $z \in \Zdz$ such that $F \ttwo z \subseteq G \tone x$, then $F$ is linearly dependent over $\Z$.

Let $F = \{f_j\}_{j=1}^d \subseteq G$ and suppose that $z \in \Zdz$ is such that $F \ttwo z \subseteq G \tone x$. There exists $\{g_j\}_{j=1}^d \subseteq G$ such that
\[\forall \ j \in \{1, \ldots, d\}, \quad \rtwo(f_j)z = f_j \ttwo z = g_j \tone x = \rone(g_j)x.\]
This means that for all $j \in \{1,\ldots, d\}$, $z = T_{f_j,g_j} x$, whereby for all $j, k \in \{1, \ldots, d\}$, $x \in \nullspace (T_{f_j,g_j} - T_{f_k,g_k})$. Since $x$ was chosen from the set in (\ref{eqn:hugegoodset}), it follows that $T_{f_j,g_j} = T_{f_k,g_k}$. This equation rearranges with the help of (\ref{eqn:expressionfortrans}) to
\begin{align}\label{eqn:inverseequalityforjk}\forall \ j, k \in \{1, \ldots, d\}, \quad \rtwo(f_k) \rtwo(f_j)^{-1} = \rone(g_k) \rone(g_j)^{-1}.\end{align}

Consider the equations in (\ref{eqn:inverseequalityforjk}) with $j=1$ fixed. By Lemma \ref{lem:reachestheid}, there exists $b \in \N$ and $\tilde{f_1}, \tilde{g_1} \in \Zdz$ such that $b\rone(g_1)^{-1} = \rone(\tilde{g_1})$ and $b \rtwo(f_1)^{-1} = \rtwo(\tilde{f_1})$. Multiplying the equations in (\ref{eqn:inverseequalityforjk}) by $b$, we see that
\[\forall \ k \in \{1, \ldots, d\}, \quad \rtwo(f_k \ttwo \tilde{f_1}) = \rone(g_k \tone \tilde{g_1}).\]
Now $\{\rtwo(f_k \ttwo \tilde{f_1})\}_{k=1}^d$ is a set of $d$ matrices contained in $\ronezd \cap \rtwozd$. Since $\tone$ and $\ttwo$ are not aligned, Lemma \ref{lem:equivconditionsforaligned} gives that $\ronezd \cap \rtwozd$ is a lattice of dimension at most $(d-1)$. Therefore, there exists $\xi \in \Zdz$ such that
\[\sum_{k=1}^d \xi_k \rtwo(f_k \ttwo \tilde{f_1}) = \rtwo \left( \left(\sum_{k=1}^d \xi_k f_k \right) \ttwo \tilde{f_1} \right) = 0.\]
Using the injectivity of $\rtwo$ and the fact that $\ttwo$ has no zero divisors, we conclude that $\sum_{k=1}^d \xi_k f_k = 0$,
meaning that $F$ is linearly dependent over $\Z$.
\end{proof}

The following two lemmas will be useful in the proof of Theorem \ref{thm:thickbutwithoutdensity}. Let $\euclid {\ \cdot \ }$ denote the Euclidean norm on $\Qd$. Lemma \ref{lem:infinitelymanyshiftsinposdensity} makes it easier to prove that a set has zero density.

\begin{lemma}\label{lem:asymptoticnormbehavior}
Let $\tone \in \discrims$ and $y \in \Zdz$. There exists a constant $K = K(y,\tone) > 0$ such that for all $x \in \Zdz$, 
\[K^{-1} \max \big(\euclid{ x \tone y }, \euclid{ y \tone x }\big) \leq \euclid{x} \leq K \min \big(\euclid{ x \tone y }, \euclid{ y \tone x }\big).\]
\end{lemma}

\begin{proof}
For all $T \in \gldq$, there exists a constant $K = K(T) > 0$ such that for all $x \in \Qdz$,
\[K^{-1}\euclid{Tx} \leq \euclid{x} \leq K\euclid{Tx}.\]
The lemma follows since $x \tone y = \roneright(x)y$ and $y \tone x = \rone(x)y$ and $\rone(x), \roneright(x) \in \gldq$.
\end{proof}

\begin{lemma}\label{lem:infinitelymanyshiftsinposdensity}
For all $\tone \in \discrims$ and $A \subseteq \Zdz$,
\begin{align}\label{eqn:infinitelymanyshiftsinposdensity}\dstar_{\tone}(A) = \sup \left\{ \alpha \geq 0 \ \middle| \ \begin{aligned} & \text{for all } F \in \finitesubsets \Zdz, \ \text{there exist infinitely} \\ & \text{many } z \in \Zdz \text{ such that } \big|(F\tone z) \cap A\big| \geq \alpha |F| \end{aligned} \right\}.\end{align}
In particular, $\dstar_{\tone}(A) = 0$ if and only if for all $\eps > 0$, there exists $F \in \finitesubsets \Zdz$ and a co-finite set $P \subseteq \Zdz$ such that for all $z \in P$, $\big|(F\tone z) \cap A\big| < \eps |F|$.
\end{lemma}

\begin{proof}
Denote temporarily the right hand side of (\ref{eqn:infinitelymanyshiftsinposdensity}) by $\dstar_{\tone,\infty}(A)$. Clearly $\dstar_{\tone,\infty}(A) \leq \dstar_{\tone}(A)$, so it suffices prove that $\dstar_{\tone,\infty}(A) \geq \al$ for all $\al < \dstar_{\tone}(A)$.

Let $\al < \dstar_{\tone}(A)$. We must show that for all $F \in \finitesubsets \Zdz$, the set
\[S_\al(F) = \big\{ z \in \Zdz \ \big| \ \big|(F\tone z) \cap A\big| \geq \alpha |F| \big\}\]
is infinite.

Let $F \in \finitesubsets \Zdz$. First, we claim that for all $z \in \Zdz$, there exists $\zeta \in \Zdz$ for which $z \tone \zeta \in S_\al(F)$. Indeed, since $\al < \dstar_{\tone}(A)$, we know $S_\al(F \tone z)$ is non-empty. For $\zeta \in S_\al(F \tone z)$,
\[\big| (F \tone z \tone \zeta ) \cap A \big| \geq \al |F \tone z| = \al |F|,\]
meaning $z \tone \zeta \in S_\al(F)$.

By Lemma \ref{lem:reachestheid}, there exist $c \in \N$ and $w \in \Zdz$ for which $\rone(w) = c \idd$. By the previous remarks, there exists $(\zeta_n)_{n \in \N} \subseteq \Zdz$ such that for all $n \in \N$,
\[c n \zeta_n = (nw) \tone \zeta_n \in S_\al(F).\]
Since $(c n \zeta_n)_{n \in \N} \subseteq S_\al(F)$ and $|\{c n \zeta_n\}_{n \in \N}| = \infty$, the set $S_\al(F)$ is infinite.
\end{proof}

\begin{theorem}\label{thm:thickbutwithoutdensity}
Let $\bigmults \subseteq \discrims$. There exists $A \subseteq \Zdz$ with the property that
\begin{enumerate}[label=(\Roman*)]
\item\label{item:reallythick} for all $\tone \in \bigmults$, $A \in \thick(\tone)$;
\item for all $\ttwo \in \discrims$ which are not aligned with any multiplication in $\bigmults$, $A \not\in \pws(\ttwo)$ and, if $(\Zdz,\ttwo)$ is left amenable, $A \not\in \density(\ttwo)$.
\end{enumerate}
\end{theorem}

\begin{proof}
It suffices to show that there exists $A \subseteq \Zdz$ satisfying \ref{item:reallythick} with the property that for all $y_1, \ldots, y_k \in \Zdz$ and all $\ttwo \in \discrims$ not aligned with $\tone$,
\begin{align}\label{eqn:killspsanddensity}d^*_{\ttwo} \big( y_1^{-1} \ttwo A \cup \ldots \cup y_k^{-1} \ttwo A \big) = 0.\end{align}
Indeed, by the definition of piecewise syndeticity and Lemma \ref{lem:basicfactsofdensity} \ref{item:basicsofthick}, equation (\ref{eqn:killspsanddensity}) shows that $A \not\in \pws(\ttwo)$. It follows from Lemma \ref{lem:basicfactsofdensity} \ref{item:inversetranslatehasmoredensity} and (\ref{eqn:killspsanddensity}) that for any $y \in \Zdz$,
\[0 \leq d^*_{\ttwo} (A) \leq d^*_{\ttwo} \big( y^{-1} \ttwo A \big) = 0.\]
Therefore, if $(\Zdz,\ttwo)$ is left amenable, then $A$ has zero density with respect to $\ttwo$.

Since $\discrims$ is countable, there exists a chain $\finitediscrims_1 \subseteq \finitediscrims_2 \subseteq \cdots \subseteq \discrims$ of finite subsets of $\discrims$ which exhaust those multiplications not aligned with any multiplication in $\bigmults$. Let $(\tone_n)_{n \in \N}$ be a sequence in $\bigmults$ which visits every element of $\bigmults$ infinitely often. Using Lemmas \ref{lem:thereareavoidingsets} and \ref{lem:asymptoticnormbehavior}, choose inductively $x_1, x_2, \ldots \in \Zdz$ so that the set $H_n = \cube n \tone_n x_n$ is $\left(\cube n, \finitediscrims_n \right)$-avoiding and satisfies $\normmin {H_{n}} > n \normmax{H_{n-1}}$, where for non-empty $F \in \finitesubsets \Zdz$ we write
\[\normmin F = \min_{f\in F}\euclid{f}, \qquad \normmax{F} = \max_{f\in F}\euclid{f}.\]
Finally, put $A = \cup_n H_n$.

By construction, $A$ is thick with respect to every multiplication in $\bigmults$. To prove (\ref{eqn:killspsanddensity}), let $y_1, \ldots, y_k \in \Zdz$ and suppose $\ttwo \in \discrims$ is not aligned with any multiplication in $\bigmults$.

\begin{claim*}\label{claim:intersectsonlyonefarout}
For all non-empty $F \in \finitesubsets \Zdz$, there exists $N_0 = N_0(\ttwo,F) \in \N$ with the property that for all $z \in \Zdz$,
\begin{align*}\big| \{ n \geq N_0(\ttwo,F) \ | \ (F \ttwo z) \cap H_n \neq \emptyset \} \big| \leq 1.\end{align*}
\end{claim*}

\begin{proof}
By Lemma \ref{lem:asymptoticnormbehavior}, there exists $N_0=N_0(\ttwo,F) \in \N$ such that for all $f \in F$ and all $z \in \Zdz$,
\[N_0^{-1} \normmax{F \ttwo z} \leq \euclid {f \ttwo z} \leq N_0 \normmin{F \ttwo z}.\]
Let $n \geq N_0$, and suppose that $(F \ttwo z) \cap H_n \neq \emptyset$. Let $f \in F$ be such that $f \ttwo z \in H_n$, and note that
\begin{align*}
n &\normmax{H_{n-1}} < \normmin{H_n} \leq \euclid {f \ttwo z} \leq N_0 \normmin{F \ttwo z},\\
(n+1)^{-1} &\normmin{H_{n+1}} > \normmax{H_n} \geq \euclid {f \ttwo z} \geq N_0^{-1} \normmax{F \ttwo z}.
\end{align*}
Since $n \geq N_0$, this means
\[\normmax{H_{n-1}} < \normmin{F \ttwo z} \leq \normmax{F \ttwo z} < \normmin{H_{n+1}}.\]
This means that the set $F \ttwo z$ is positioned between shells containing $H_{n-1}$ and $H_{n+1}$; in particular, it can have non-empty intersection with $H_n$ only.
\end{proof}

To show (\ref{eqn:killspsanddensity}), it suffices by Lemma \ref{lem:infinitelymanyshiftsinposdensity} to show that for all $\eps > 0$, there exists $F \in \finitesubsets \Zdz$ and a cofinite $P \subseteq \Zdz$ such that for all $z \in P$,
\begin{align}\label{eqn:requiredsmallintersection}
\big| (F \ttwo z) \cap (y_1^{-1} \ttwo A \cup \ldots \cup y_k^{-1} \ttwo A) \big| < \eps |F|.
\end{align}

Let $\eps > 0$, and let $F \in \finitesubsets \Zdz$ satisfy $|F| > k d \eps^{-1}$ and be such that no subset of $d$ points is linearly dependent over $\Z$. Let
\[N \geq \max \left( \big\{\normmax{y_i \ttwo F} \big\}_{i=1}^k \cup \big\{ N_0(\ttwo,y_i \ttwo F) \big\}_{i=1}^k \right)\]
be sufficiently large so that $\ttwo \in \finitediscrims_N$, and set
\[P = \left\{ z \in \Zdz \ \middle| \ \text{for all } 1 \leq i \leq k, \ \normmin{y_i \ttwo F \ttwo z} > \normmax{H_N} \right\}.\]
By Lemma \ref{lem:asymptoticnormbehavior}, the set $P$ is co-finite. Note that by the definition of $P$, if $z \in P$ and $(y_i \ttwo F \ttwo z) \cap H_n \neq \emptyset$, then $n > N$.

To show (\ref{eqn:requiredsmallintersection}), let $z \in P$ and $1 \leq i \leq k$. The left hand side of (\ref{eqn:requiredsmallintersection}) is bounded from above by
\[\sum_{i=1}^k \big|(F \ttwo z) \cap (y_i^{-1} \ttwo A)\big| = \sum_{i=1}^k \big|(y_i \ttwo F \ttwo z) \cap A \big|,\]
where the equality follows by left cancellativity. Therefore, it suffices to bound each term in the sum on the right hand side by $|F| \eps \big/ k$.

Fix $i \in \{1, \ldots, k\}$. If $(y_i \ttwo F \ttwo z) \cap A \neq \emptyset$, then by the definition of $P$ and the previous claim, there is an $n > N$ for which
\[(y_i \ttwo F \ttwo z) \cap A = (y_i \ttwo F \ttwo z) \cap H_n.\]
Since $n > N$, the set $y_i \ttwo F \subseteq \cube n$, and $H_n$ is $(\cube n, \ttwo)$-avoiding. Because no $d$ points of $F$ are linearly dependent over $\Z$ and $\ttwo$ has no zero divisors, no $d$ points of $y_i \ttwo F$ are linearly dependent over $\Z$. It follows by Definition \ref{def:avoiding} that
\[|(y_i \ttwo F \ttwo z) \cap A| = |(y_i \ttwo F \ttwo z) \cap H_n| \leq d-1 < |F| \eps \big/ k,\]
completing the proof of (\ref{eqn:requiredsmallintersection}) and the theorem.
\end{proof}

\begin{corollary}\label{cor:noclasscontainment}
Let $\arbclass \in \{\syndetic,\thick,\pws,\pws^*,\density,\density^*,\central,\central^*\}$. If $\tone, \ttwo \in \discrims$ are not aligned, then $\arbclass(\tone) \not\subseteq \arbclass(\ttwo)$ and $\arbclass(\ttwo) \not\subseteq \arbclass(\tone)$.
\end{corollary}

\begin{proof}
Theorem \ref{thm:thickbutwithoutdensity} gives that $\thick(\tone) \not\subseteq \pws(\ttwo)$. By Lemma \ref{lem:hierarchylemma}, $\thick(\tone) \subseteq \central(\tone)$ and $\central(\ttwo) \subseteq \pws(\ttwo)$, so $\central(\tone) \not\subseteq \central(\ttwo)$, whereby $\central^*(\ttwo) \not\subseteq \central^*(\tone)$. Similarly, $\thick(\tone) \subseteq \pws(\tone)$, so $\pws(\tone) \not\subseteq \pws(\ttwo)$, whereby $\pws^*(\ttwo) \not\subseteq \pws^*(\tone)$.  Again, since $\thick(\ttwo) \subseteq \pws(\ttwo)$, $\thick(\tone) \not\subseteq \thick(\ttwo)$, whereby $\syndetic(\ttwo) \not\subseteq \syndetic(\tone)$. If both $(\Zdz,\tone)$ and $(\Zdz,\ttwo)$ are left amenable, the previous statements hold with $\pws$ replaced by $\density$ and $\pws^*$ replaced by $\density^*$. The remaining statements follow by interchanging $\tone$ and $\ttwo$.
\end{proof}

Theorem \ref{thm:equaldensityfunctions} and Corollary \ref{cor:noclasscontainment} combine to complete the proof of Theorem A from the introduction.

The result in Theorem \ref{thm:thickbutwithoutdensity} can be extended in certain cases. Recall the notation from Section \ref{sec:earlyexamples}. When $c$ is positive, one can prove the following theorem by exploiting matrices in $\rthree {x^2-c}(\Ztz)$ with eigenvalues of absolute value less than and greater than $1$.

\begin{theorem}
Let $C \subseteq \N \setminus \{1^2,2^2,\ldots\}$. There exists $A \subseteq \Ztz$ with the property that
\begin{enumerate}[label=(\Roman*)]
\item for all $c \in C$, $A \in \density^*(\tthree {x^2-c})$;
\item for all $c \in \N \setminus C$ not a square, $A \not\in \density(\tthree {x^2-c})$.
\end{enumerate}
\end{theorem}

Whether such an extension of Theorem \ref{thm:thickbutwithoutdensity} is always possible remains unknown. We were unable to prove or disprove, for example, the analogous result when $c$ is allowed to be negative. Concretely: Is there a set which is $\density^*$ with respect to the multiplication arising from $\Z[i]$ but of zero density with respect to the multiplication arising from $\Z[\sqrt{-2}]$?

The technique used in the proof of Theorem \ref{thm:thickbutwithoutdensity} can be used to improve \cite[Theorem 3.6]{berghind-onipsets} to this setting. The following theorem shows that \adive{} $\ip^*$ sets need not be as multiplicatively large as what is guaranteed by Theorem \ref{thm:thmfrombgotherpaper} for \adive{} $\ipr^*$ sets.

\begin{theorem}\label{thm:addipdoesnotimplymultsyndetic}
There exists a set $A \subseteq \Zd$ which is \adly{} $\ip^*$ but for which $A \setminus \{0\}$ is not multiplicatively syndetic with respect to any proper multiplication on $\Zd$.
\end{theorem}

\begin{proof}
Let $(\tone_n)_{n \in \N}$ be a sequence in $\discrims$ which visits every element of $\discrims$ infinitely often. Using Lemma \ref{lem:asymptoticnormbehavior}, choose inductively $x_1, x_2, \ldots \in \Zdz$ so that $\normmin{\cube {2n} \tone_n x_n} > n$ and the set $H_n = \cube n \tone_n x_n$ satisfies $\normmin {H_{n}} > 2 \normmax{H_{n-1}}$. Put $B = \cup_n H_n$ and $A = \Zd \setminus B$.

By construction, the set $B$ is multiplicatively thick with respect to all multiplications in $\discrims$, so $A \setminus \{0\} = \Zdz \setminus B$ is not multiplicatively syndetic with respect to any multiplication in $\discrims$. To prove that $A$ is additively $\ip^*$, we need only to show that $B$ is not additively $\ip$.

To prove that $B$ is not an additive $\ip$ set, it suffices to show that for all $x \in \Zdz$, $\big| B \cap (B-x) \big| < \infty$. Let $x \in \Zdz$, and choose $N \in \N$ such that $N \geq \euclid{x}$ and $\normmin{H_N} > \euclid{x}$. To prove that $\big| B \cap (B-x) \big| < \infty$, it suffices to prove that for all $n, m \in \N$, if $n \neq m$ or $n \geq N$, then the set $H_n \cap (H_m - x)$ is empty.

Let $n, m \in \N$ and suppose $H_n \cap (H_m - x)$ is non-empty. This means
\[\normmin{H_n} \leq \normmax{H_m - x} \leq 2\normmax{H_m} < \normmin{H_{m+1}}.\]
It follows that $n < m+1$. By a similar argument, $m < n+1$. Therefore, $n=m$. Since $H_n \cap (H_n - x)$ is non-empty, $x$ is an element of the difference set $H_n - H_n = (\cube n - \cube n) \tone_n x_n$. Since $x \neq 0$, $x \in \cube {2n} \tone_n x_n$, whereby $n < \normmin{ \cube {2n} \tone_n x_n} \leq \euclid{x}$. Since $\euclid{x} \leq N$, this proves that $n < N$.
\end{proof}

While \adive{} $\ip^*$ sets need not be multiplicatively syndetic, they are multiplicatively thick with respect to all proper multiplications; this follows from the dual statement to \cite[Theorem 6.2]{BGpaperone}.

\section{Proof of Theorem B}\label{sec:maintheoremtwo}  

We will prove Theorem B in two parts, beginning with the ``if'' direction. Let $\tone, \ttwo \in \discrims$. If $\tone=\ttwo$, then all of the corresponding classes of largeness coincide. If $\tone = \ttwoop$, then for all $r \in \N$ and $(x_n)_{n=1}^r \subseteq \Zdz$,
\begin{align}\label{eqn:reversingtrick}\finiteproducts_\ttwo (x_1, x_2, \ldots, x_r) = \finiteproducts_{\tone} (x_r, x_{r-1}, \ldots, x_1).\end{align}
This means that for all $\arbclass \in \{\iprclass,\ipnaughtclass, \iprclass^*, \ipnaughtclass^*\}$, $\arbclass(\tone) = \arbclass(\ttwo)$, proving the ``if'' direction in Theorem B from the introduction.

Next we prove that for $r \geq 2$, the class of multiplicative $\ipr$ sets determines the multiplication up to the opposite operation. We accomplish this by assuming that $\ttwo \not\in \{\tone, \toneop\}$ and constructing a set which is IP with respect to $\tone$ but which contains no solutions to the equation $x \ttwo y = z$.

We will make use of the following notation: for non-empty $\al = \{\al_1 < \cdots < \al_k\} \in \finitesubsets \N$,
\[x_\al = x_{\al_1} \tone \cdots \tone x_{\al_k}.\]
The multiplication $\tone$ is suppressed in this notation. Even though we allow $\emptyset \in \finitesubsets \N$, because $(\Zd,+,\tone)$ may not have a multiplicative identity, we leave the symbol $x_\emptyset$ undefined.

\begin{theorem}\label{thm:ipdeterminesmult}
Let $\tone \in \discrims$, and suppose $\ttwo \in \discrims \setminus \{\tone, \toneop\}$. There exists $A \subseteq \Zdz$ which is IP with respect to $\tone$ but not IP$_2$ with respect to $\ttwo$.
\end{theorem}

\begin{proof}
We will construct a sequence $(x_n)_{n \in \N} \subseteq \Zdz$ by induction such that for all non-empty $\al, \be, \ga \in \finitesubsets \N$, the equation
\[E(\al,\be,\ga): \ x_\al \ttwo x_\be = x_\ga\]
is false. The set $A = \finiteproducts_\tone(x_n)_{n \in \N}$ will then satisfy the conclusions of the theorem.

To construct such a sequence, we must also consider the equations
\begin{align*}
F(\al,\be,\ga): \ &\rtwo(x_\al) \rone(x_\be) = \rone(x_\ga), & G(\al): \ &\rtwo = \rone(x_\al) \circ \rone,\\
F_{\text{r}}(\al,\be,\ga): \ &\rtworight(x_\al) \rone(x_\be) = \rone(x_\ga), & G_{\text{r}}(\al): \ &\rtworight = \rone(x_\al) \circ \rone,
\end{align*}
where $\rone$, $\rtwo$ are the representations of $\tone$, $\ttwo$, respectively, and $\rone(x_\emptyset)$ and $\rtwo(x_\emptyset)$ stand for the identity matrix $\idd$.

Call $E(\al,\be,\ga)$, $F(\al,\be,\ga)$, and $F_{\text{r}}(\al,\be,\ga)$ \emph{homogeneous in $x_n$} if $n \not\in \al \cup \be \cup \ga$ or $n \in (\al \cap \ga) \bigtriangleup (\be \cap \ga)$; call $G(\al)$ and $G_{\text{r}}(\al)$ \emph{homogeneous in $x_n$} if $n \not\in \al$. By linearity, the truth of equations which are homogeneous in $x_n$ remains invariant under the transformation $x_n \mapsto c x_n$ for $c \in \N$. In contrast, given a finite collection of equations which are not homogeneous in $x_n$ and whose truths have been determined, there exists $c \in \N$ for which all equations in the collection become false when $x_n$ is replaced by $cx_n$.

We proceed now to construct $(x_n)_{n \in \N} \subseteq \Zdz$ inductively so that for all $n \in \N$, the following statements hold:
\begin{align*}
\condone_n: \ & \text{for all non-empty $\al, \be, \ga \in \finitesubsets {\{1, \ldots, n\}}$, $E(\al,\be,\ga)$ is false;}\\
\condtwo_n: \ & \text{for all $\al, \be, \ga \in \finitesubsets {\{1, \ldots, n\}}$ with $\al \neq \emptyset$, $F(\al,\be,\ga)$ and $F_{\text{r}}(\al,\be,\ga)$ are false;}\\
\condthree_n: \ & \text{for all $\al \in \finitesubsets {\{1, \ldots, n\}}$, $G(\al)$ and $G_{\text{r}}(\al)$ are false.}
\end{align*}

\textbf{Base case:} By the comments above, it suffices to find $x_1$ satisfying the statements involving equations which are homogeneous in $x_1$. Thus, from $\condtwo_1$, we need $x_1$ to satisfy $\rtwo(x_1) \neq \rone(x_1)$ and $\rtworight(x_1) \neq \rone(x_1)$, and from $\condthree_1$, we need $\rone \not\in \{\rtwo, \rtworight\}$. The latter follows from our assumption that $\ttwo \not\in \{\tone, \toneop\}$, while the former is satisfied by any $x_1$ in the non-empty set $\Zdz \setminus \big(\nullspace (\rtwo-\rone) \cup \nullspace (\rtworight-\rone)\big)$.

\textbf{Inductive step:} Suppose $x_1, \ldots, x_{n-1} \in \Zdz$ have been chosen so that $\condone_{n-1}$, $\condtwo_{n-1}$, and $\condthree_{n-1}$ hold. Let
\begin{align*}
\nullcoll_1 &= \bigcup_{\substack{\al, \be, \ga \\ \al \neq \emptyset}} \Big( \nullspace \big( \rtwo(x_\al) \rone(x_\be) - \rone(x_\ga) \big) \cup \nullspace \big( \rtworight(x_\al) \rone(x_\be) - \rone(x_\ga) \big) \Big),\\
\nullcoll_2 &= \bigcup_{\al, \be, \ga} \Big( \nullspace \big( z \mapsto \rtwo (\rone(x_\al) z)\rone(x_\be) - \rone(x_\ga)\rone(z) \big)\\
& \qquad \qquad \qquad \qquad \quad \cup \nullspace \big( z \mapsto \rtworight (\rone(x_\al) z)\rone(x_\be) - \rone(x_\ga)\rone(z) \big) \Big),
\end{align*}
where each of the unions is over $\al, \be, \ga \in \finitesubsets {\{1, \ldots, n-1\}}$. Put $\nullcoll = \nullcoll_1 \cup \nullcoll_2$. We will show that $\Zdz \setminus \nullcoll$ is non-empty and that any $x_n \in \Zdz \setminus \nullcoll$ satisfies the statements in $\condone_{n}$, $\condtwo_{n}$, and $\condthree_{n}$ involving equations homogeneous in $x_n$. As discussed above, by replacing $x_n$ with $c x_n$ for some $c \in \N$, this suffices to complete the induction.

To see that $\Zdz \setminus \nullcoll$ is non-empty, it suffices to show that each of the $\Z$-linear transformations involved in the definitions of $\nullcoll_1$ and $\nullcoll_2$ is not identically zero. For $\nullcoll_1$, since $\al, \be, \ga \in \finitesubsets {\{1, \ldots, n-1\}}$ with $\al \neq \emptyset$, it follows immediately from $\condtwo_{n-1}$ that $\rtwo(x_\al) \rone(x_\be) - \rone(x_\ga) \neq 0$ and $\rtworight(x_\al) \rone(x_\be) - \rone(x_\ga) \neq 0$.

For $\al, \be, \ga \in \finitesubsets {\{1, \ldots, n-1\}}$, let $\newtran(z) = \rtwo (\rone(x_\al) z)\rone(x_\be) - \rone(x_\ga)\rone(z)$, one of the transformations appearing in $\nullcoll_2$. We wish to show that $\newtran$ is non-zero as a $\Z$-linear transformation from $\Zd$ to $\matdz$. Consider the following cases.
\begin{enumerate}[label=(\Roman*)]
\item $\al \neq \emptyset$. Using Lemma \ref{lem:reachestheid}, there exist $c \in \N$ and $w \in \Zdz$ such that $\rone(w) = \roneright(w) = c \idd$. Note that $\newtran(w) = c \big( \rtwo(x_\al) \rone(x_\be) - \rone(x_\ga) \big) \neq 0$ by $\condtwo_{n-1}$, meaning $\sigma$ is not identically zero.
\item $\al = \emptyset$, $\be \neq \emptyset$. Note that $\newtran(x_\be) = \big(\rtwo (x_\be) - \rone(x_\ga) \big) \rone(x_\be) \neq 0$ by $\condtwo_{n-1}$, meaning $\sigma$ is not identically zero.
\item $\al = \be = \emptyset$, $\ga \neq \emptyset$. In this case, $\newtran(z) = \rtwo(z) - \rone(x_\ga)\rone(z)$ is not identically zero by $\condthree_{n-1}$.
\item $\al = \be = \ga = \emptyset$. In this case, $\sigma(z) = \rtwo(z) - \rone(z)$. This is not identically zero since $\rtwo \neq \rone$ by assumption.
\end{enumerate}
It can be shown in the same way that the other $\Z$-linear transformations in $\nullcoll_2$ are not identically zero. This shows that $\nullcoll$ is a finite collection of proper linear subspaces of $\Zd$, whereby $\Zdz \setminus \nullcoll$ is non-empty.

Let $x_n \in \Zdz \setminus \nullcoll$. We wish to show that the statements in $\condone_{n}$, $\condtwo_{n}$, and $\condthree_{n}$ involving equations homogeneous in $x_n$ hold. There are three cases to consider for an equation $E(\al,\be,\ga)$ from $\condone_n$ which is homogeneous in $x_n$:
\begin{enumerate}[label=(\Roman*)]
\item $n \not\in \al\cup\be\cup\ga$. $E(\al,\be,\ga)$ is false by $\condone_{n-1}$.
\item $n \in \al \cap \ga$, $n \not\in \be$. $E(\al,\be,\ga)$ can be written as $\rtworight(x_\be) \rone(x_{\al_0})x_n = \rone(x_{\ga_0})x_n$, where $\al_0, \be, \ga_0 \in \finitesubsets {\{1, \ldots, n-1\}}$ and $\be \neq \emptyset$. Since $x_n \not\in \nullcoll_1$, $E(\al,\be,\ga)$ is false.
\item $n \in \be \cap \ga$, $n \not\in \al$. $E(\al,\be,\ga)$ can be written as $\rtwo(x_\al) \rone(x_{\be_0})x_n = \rone(x_{\ga_0})x_n$, where $\al, \be_0, \ga_0 \in \finitesubsets {\{1, \ldots, n-1\}}$ and $\al \neq \emptyset$. Since $x_n \not\in \nullcoll_1$, $E(\al,\be,\ga)$ is false.
\end{enumerate}
This shows that the statements in $\condone_{n}$ involving equations homogeneous in $x_n$ hold.

There are three cases to consider for an equation $F(\al,\be,\ga)$ from $\condtwo_n$ which is homogeneous in $x_n$:
\begin{enumerate}[label=(\Roman*)]
\item $n \not\in \al\cup\be\cup\ga$. $F(\al,\be,\ga)$ is false by $\condtwo_{n-1}$.
\item $n \in \al \cap \ga$, $n \not\in \be$. $F(\al,\be,\ga)$ can be written as $\rtwo(\rone(x_{\al_0})x_n) \rone(x_{\be}) = \rone(x_{\ga_0})\rone(x_n)$, where $\al_0, \be, \ga_0 \in \finitesubsets {\{1, \ldots, n-1\}}$. Since $x_n \not\in \nullcoll_2$, $F(\al,\be,\ga)$ is false.
\item $n \in \be \cap \ga$, $n \not\in \al$. $F(\al,\be,\ga)$ can be written as $\rtwo(x_{\al}) \rone(x_{\be_0})\rone(x_n) = \rone(x_{\ga_0})\rone(x_n)$, where $\al, \be_0, \ga_0 \in \finitesubsets {\{1, \ldots, n-1\}}$ and $\al \neq \emptyset$. Since $\rone(x_n) \neq 0$, $F(\al,\be,\ga)$ is false by $\condtwo_{n-1}$.
\end{enumerate}
It can be shown in the same way that the equations $F_{\text{r}}(\al,\be,\ga)$ from $\condtwo_{n}$ which are homogeneous in $x_n$ are false.  This shows that the statements in $\condtwo_{n}$ involving equations homogeneous in $x_n$ hold.

Consider an equation $G(\al)$ from $\condthree_n$ which is homogeneous in $x_n$. It must be that $n \not\in \al$, and so $G(\al)$ is false by $\condthree_{n-1}$. This shows that the statements in $\condthree_{n}$ involving equations homogeneous in $x_n$ hold.

This completes the proof of the inductive step and the proof of the theorem.
\end{proof}

The following theorem shows that the ``reversal'' trick in (\ref{eqn:reversingtrick}) does not work for finite product sets with infinitely many generators. We need some more notation: for non-empty $\al, \be \in \finitesubsets \N$, we write $\al < \be$ to mean that $\max \al < \min \be$.

\begin{theorem}
Suppose $\tone \in \discrims$ is non-commutative, that is, $\tone \neq \toneop$. There exists a set which is IP with respect to $\tone$ but not IP with respect to $\toneop$.
\end{theorem}

\begin{proof}
It suffices to find a sequence $(x_n)_{n \in \N} \subseteq \Zdz$ with the property that for all non-empty $\al, \be, \ga \in \finitesubsets \N$, if $x_\al \tone x_\beta = x_\gamma$, then $\al < \be$. Indeed, we claim that $A = \finiteproducts_\tone(x_n)_{n \in \N}$ is not an IP set with respect to $\toneop$. Suppose for a contradiction that there exists $(y_n)_{n \in \N} \subseteq \Zdz$ for which $\finiteproducts_{\toneop}(y_n)_{n \in \N} \subseteq A$. For each $n \in \N$, let $\al_n$ be a non-empty, finite subset of $\N$ for which $y_n = x_{\al_n}$. Since
\[x_{\al_2} \tone x_{\al_1} = y_1 \toneop y_2 \in \finiteproducts_{\toneop}(y_n)_{n \in \N} \subseteq A,\]
there exists a non-empty $\ga \in \finitesubsets \N$ such that $x_{\al_2} \tone x_{\al_1} = x_\ga$, meaning $\al_2 < \al_1$. Repeating this argument for general $y_n$'s, we see that $\al_1 > \al_2 > \cdots$ is an infinite, strictly decreasing chain of non-empty subsets of $\N$. This is clearly impossible, yielding a contradiction.

As in the proof of Theorem \ref{thm:ipdeterminesmult}, we require an induction hypothesis which is stronger than the desired conclusion. We must consider the equations
\begin{align*}
S(\al,\be,\ga): \ &x_\al \tone x_\be = x_\ga, & U(\al,\be): \ &\rone(x_\al)\rone(x_\be) = \idd,\\
T(\al,\be,\ga): \ &\rone(x_\al)\roneright(x_\be) = \rone(x_\ga),
\end{align*}
where $\rone$ is the representation of $\tone$ and $\rone(x_\emptyset)$ stands for the identity matrix $\idd$.

Call $S(\al,\be,\ga)$ and $T(\al,\be,\ga)$ \emph{homogeneous in $x_n$} if $n \not\in \al \cup \be \cup \ga$ or $n \in (\al \cap \ga) \bigtriangleup (\be \cap \ga)$; call $U(\al,\be)$ \emph{homogeneous in $x_n$} if $n \not\in \al \cup \be$. The comments made about homogeneity at this point in the proof of Theorem \ref{thm:ipdeterminesmult} apply here, too.

We proceed now to construct $(x_n)_{n \in \N} \subseteq \Zdz$ inductively so that for all $n \in \N$, the following statements hold:
\begin{align*}
\condone_n: \ & \text{for all non-empty $\al, \be, \ga \in \finitesubsets {\{1, \ldots, n\}}$, if $S(\al,\be,\ga)$ is true, then $\al < \be$;}\\
\condtwo_n: \ & \text{for all $\al, \be, \ga \in \finitesubsets {\{1, \ldots, n\}}$ with $\be \neq \emptyset$, $T(\al,\be,\ga)$ is false;}\\
\condthree_n: \ & \text{for all non-empty $\al, \be \in \finitesubsets {\{1, \ldots, n\}}$, $U(\al,\be)$ is false.}
\end{align*}

\textbf{Base case:} It suffices to find $x_1$ satisfying the statements involving equations which are homogeneous in $x_1$. Thus, we only need $x_1$ to satisfy $\roneright(x_1) \neq \rone(x_1)$. Since $\tone$ is non-commutative, $\roneright \neq \rone$, so any $x_1$ in the non-empty set $\Zdz \setminus \nullspace (\roneright-\rone)$ will do.

\textbf{Inductive step:} Suppose $x_1, \ldots, x_{n-1} \in \Zdz$ have been chosen so that $\condone_{n-1}$, $\condtwo_{n-1}$, and $\condthree_{n-1}$ hold. Let
\begin{align*}
\nullcoll_1 &= \bigcup_{\substack{\al, \be, \ga \\ \be \neq \emptyset}} \nullspace \big( \rone(x_\al) \roneright(x_\be) - \rone(x_\ga) \big),\\
\nullcoll_2 &= \bigcup_{\al, \be, \ga} \nullspace \big( z \mapsto \rone (x_\al) \roneright(z)\roneright(x_\be) - \rone(x_\ga)\rone(z) \big),\\
\nullcoll_3 &= \bigcup_{\substack{\al, \be, \ga \\ \al, \be \neq \emptyset}} \nullspace \big( \rone(x_\al) \rone(x_\be) - \idd \big),
\end{align*}
where each of the unions is over $\al, \be, \ga \in \finitesubsets {\{1, \ldots, n-1\}}$. Put $\nullcoll = \nullcoll_1 \cup \nullcoll_2 \cup \nullcoll_3$. We will show that $\Zdz \setminus \nullcoll$ is non-empty and that any $x_n \in \Zdz \setminus \nullcoll$ satisfies the statements in $\condone_{n}$, $\condtwo_{n}$, and $\condthree_{n}$ involving equations homogeneous in $x_n$. As explained before, by replacing $x_n$ with $c x_n$ for some $c \in \N$, this suffices to complete the induction.

To see that $\Zdz \setminus \nullcoll$ is non-empty, it suffices to show that each of the $\Z$-linear transformations involved in the definitions of $\nullcoll_1$, $\nullcoll_2$, and $\nullcoll_3$ is not identically zero. For the linear equations in $\nullcoll_1$ and $\nullcoll_3$, this follows immediately from $\condtwo_{n-1}$ and $\condthree_{n-1}$, respectively.

For $\al, \be, \ga \in \finitesubsets {\{1, \ldots, n-1\}}$, let $\newtran(z) = \rone (x_\al) \roneright(z)\roneright(x_\be) - \rone(x_\ga)\rone(z)$, one of the transformations appearing in $\nullcoll_2$. We wish to show that $\newtran$ is non-zero as a $\Z$-linear transformation from $\Zd$ to $\matdz$. Using Lemma \ref{lem:reachestheid}, there exist $c \in \N$ and $w \in \Zdz$ such that $\rone(w) = \roneright(w) = c \idd$. Note that $\sigma(w) = c \big(\rone (x_\al)\roneright(x_\be) - \rone(x_\ga) \big)$. If $\be \neq \emptyset$, then $\sigma(w) \neq 0$ by $\condtwo_{n-1}$, so $\sigma$ is not identically zero. On the other hand, if $\be = \emptyset$ and $\sigma$ is identically zero, then $\sigma(w) = 0$, meaning $\rone (x_\al) = \rone(x_\ga)$. This would mean that $\sigma(z) = \rone (x_\al) \big( \roneright(z) - \rone(z) \big)$ is identically zero, whereby $\roneright = \rone$, contradicting the assumption that $\tone$ is not commutative. This shows that $\nullcoll_2$, hence $\nullcoll$, is a finite collection of proper linear subspaces of $\Zd$. Therefore, $\Zdz \setminus \nullcoll$ is non-empty.

Let $x_n \in \Zdz \setminus \nullcoll$. We wish to show that the statements in $\condone_{n}$, $\condtwo_{n}$, and $\condthree_{n}$ involving equations homogeneous in $x_n$ hold. There are three cases to consider for an equation $S(\al,\be,\ga)$ from $\condone_n$ which is homogeneous in $x_n$:
\begin{enumerate}[label=(\Roman*)]
\item $n \not\in \al\cup\be\cup\ga$. $\condone_n$ holds by $\condone_{n-1}$.
\item $n \in \al \cap \ga$, $n \not\in \be$. $S(\al,\be,\ga)$ can be written as $\rone(x_{\al_0})\roneright(x_\be)x_n = \rone(x_{\ga_0})x_n$, where $\al_0, \be, \ga_0 \in \finitesubsets {\{1, \ldots, n-1\}}$ and $\be \neq \emptyset$. Since $x_n \not\in \nullcoll_1$, the equation $S(\al,\be,\ga)$ is false, so $\condone_n$ holds.
\item $n \in \be \cap \ga$, $n \not\in \al$. Let $\al, \be_0, \ga_0 \in \finitesubsets {\{1, \ldots, n-1\}}$ be such that $\be = \be_0 \cup \{n\}$ and $\ga = \ga_0 \cup \{n\}$. Note that $\al \neq \emptyset$. Consider the following cases.
\begin{enumerate}[label=(\roman*)]
\item $\be_0= \emptyset$. Since $\be = \{n\}$, if $S(\al,\be,\ga)$ is true, then $\al < \be$.
\item $\be_0 \neq \emptyset, \ga_0 = \emptyset$. $S(\al,\be,\ga)$, which can be written as $\rone(x_\al)\rone(x_{\be_0})x_n = x_n$, is false because $x_1 \not\in \nullcoll_3$. Thus, $\condone_{n}$ holds.
\item $\be_0 \neq \emptyset, \ga_0 \neq \emptyset$. By cancellativity, $S(\al,\be,\ga)$ can be written as $x_\al \tone x_{\be_0} = x_{\ga_0}$. If true, then $\condone_{n-1}$ gives that $\al < \be_0$. This means $\al < \be$, so $\condone_{n}$ holds.
\end{enumerate}
\end{enumerate}
This shows that the statements in $\condone_{n}$ involving equations homogeneous in $x_n$ hold.

There are three cases to consider for an equation $T(\al,\be,\ga)$ from $\condtwo_n$ which is homogeneous in $x_n$:
\begin{enumerate}[label=(\Roman*)]
\item $n \not\in \al\cup\be\cup\ga$. $T(\al,\be,\ga)$ is false by $\condtwo_{n-1}$.
\item $n \in \al \cap \ga$, $n \not\in \be$. $T(\al,\be,\ga)$ can be written as $\rone(x_{\al_0}) \rone(x_n) \roneright(x_{\be}) = \rone(x_{\ga_0})\rone(x_n)$, where $\al_0, \be, \ga_0 \in \finitesubsets {\{1, \ldots, n-1\}}$ and $\be \neq \emptyset$. Since $\rone(x_n)$ and $\roneright(x_{\be})$ commute and $\rone(x_n)$ is invertible, $F(\al,\be,\ga)$ is false by $\condtwo_{n-1}$.
\item $n \in \be \cap \ga$, $n \not\in \al$. $T(\al,\be,\ga)$ can be written as $\rone(x_{\al}) \roneright(x_n)\roneright(x_{\be_0}) = \rone(x_{\ga_0})\rone(x_n)$, where $\al, \be_0, \ga_0 \in \finitesubsets {\{1, \ldots, n-1\}}$. Since $x_n \not\in \nullcoll_2$, $T(\al,\be,\ga)$ is false.
\end{enumerate}
This shows that the statements in $\condtwo_{n}$ involving equations homogeneous in $x_n$ hold.

If $U(\al,\be)$ from $\condthree_n$ is homogeneous in $x_n$, then $n \not\in \al \cup \be$, and so $U(\al,\be)$ is false by $\condthree_{n-1}$. This shows that the statements in $\condthree_{n}$ involving equations homogeneous in $x_n$ hold.

This completes the proof of the inductive step and the proof of the theorem.
\end{proof}

The proof of the following corollary follows in the same way as the proof of Corollary \ref{cor:noclasscontainment} using Lemma \ref{lem:hierarchylemma} and is omitted.

\begin{corollary}\label{cor:conclusionofmaintheormtwo}
Let $r \geq 2$ and $\tone, \ttwo \in \discrims$. For all $\arbclass \in \{\ipclass,\ipclass^*\}$, if $\tone \neq \ttwo$, then
\begin{align}\label{eqn:classesingeneralposition}\arbclass(\tone) \not\subseteq \arbclass(\ttwo) \text{ and } \arbclass(\ttwo) \not\subseteq \arbclass(\tone).\end{align}
Moreover, for all $\arbclass \in \{\iprclass,\iprclass^*,\ipnaughtclass,\ipnaughtclass^*\}$, if $\ttwo \not\in \{\tone,\toneop\}$, then (\ref{eqn:classesingeneralposition}) holds.
\end{corollary}

The remarks at the beginning of this section combine with Corollary \ref{cor:conclusionofmaintheormtwo} to complete the proof of Theorem B from the introduction.

We can improve Theorem \ref{thm:ipdeterminesmult} in certain cases. Recall the notation from Section \ref{sec:earlyexamples}. It is easy to see that the set \[\left\{n \in \Zz \ \middle | \ \text{the 2-adic valuation of $n$ is even} \right\},\]
is in $\ipclass_2^*(\tone_{[1]})$ but not in $\ipclass_2(\tone_{[2]})$. Using Lemma \ref{lem:hierarchylemma}, this shows in particular that the classes $\central(\tone_{[1]})$ and $\central(\tone_{[2]})$ are in general position. For a general pair of aligned multiplications $\tone, \ttwo \in \discrims$, the relationship between the classes $\central(\tone)$ and $\central(\ttwo)$  remains to be better understood.

\section{Proof of Corollary C}\label{sec:proofofcorollary} 

Let $\tone \in \discrims$ and $\rone$ be its representation. It is not hard to show that in a left cancellative semigroup, the classes $\syndetic$, $\thick$, $\pws$, $\pws^*$, $\density$, and $\density^*$ are all left translation invariant. Therefore, for all $T \in \rone(\Zdz)$,
\begin{align}\label{eqn:preservestheclass}T \big( \arbclass(\tone) \big) = \{T A \ | \ A \in \arbclass(\tone)\} \subseteq \arbclass(\tone).\end{align}
When (\ref{eqn:preservestheclass}) holds, we say that \emph{$T$ preserves the class $\arbclass(\tone)$}. In this section, we determine exactly which invertible linear transformations preserve the classes of largeness with respect to $\tone$ defined in Section \ref{sec:defs}.

Recall the $\gldq$-action on $\qdiscrims$ from Section \ref{sec:ringreps}. The main utility of this action comes from the fact that if $\tone, \tone_T \in \discrims$, then
\[T: (\Zdz,\tone_T) \longrightarrow (\Zdz,\tone)\]
is a semigroup homomorphism which preserves multiplicative largeness, as shown in the following lemma. This lemma will allow us to relate $TA$ with the classes $\syndetic(\tone)$ and $\syndetic(\tone_T)$. 

\begin{lemma}\label{lem:semigrouphomomorphism}
Let $r \in \N$, $\tone \in \discrims$, $T \in \matdz \cap \gldq$, and suppose $\tone_T \in \discrims$. The map
\[T: (\Zdz,\tone_T) \longrightarrow (\Zdz,\tone)\]
is a semigroup homomorphism with the property that for all $A \subseteq \Zdz$ and all $\arbclass \in \{\syndetic,\thick,\pws,\pws^*,\density,\density^*, \iprclass,\ipnaughtclass,\ipclass\}$,
\begin{align}\label{eqn:inclassifandonlyif}A \in \arbclass(\tone_T) \text{ if and only if } TA \in \arbclass(\tone).\end{align}
If $T \in \gldz$, then (\ref{eqn:inclassifandonlyif}) holds for all $\arbclass \in \{\iprclass^*, \ipnaughtclass^*, \ipclass^*\}$ as well.
\end{lemma}

\begin{proof}
It is straightforward to check that $T$ is a semigroup homomorphism. Since $T$ is injective, it is straightforward to verify that (\ref{eqn:inclassifandonlyif}) holds for the classes $\iprclass$, $\ipnaughtclass$, and $\ipclass$.  If $T \in \gldz$, then $T$ is a semigroup isomorphism, in which case the conclusion holds for all classes, in particular $\iprclass^*$, $\ipnaughtclass^*$, and $\ipclass^*$.

For the remaining classes, first we will prove that for all $A \subseteq \Zdz$,
\begin{align}\label{eqn:equaldensities}\dstar_{\tone_T}(A) = \dstar_\tone(TA).\end{align}
To this end, let $\al < \dstar_{\tone_T}(A)$, and let $F \in \finitesubsets \Zdz$. Let $c = \det(T)$, and note that $G = cT^{-1}F \in \finitesubsets \Zdz$. Since $\al < \dstar_{\tone_T}(A)$, there exists $x \in \Zdz$ for which
\begin{align*}
\big|\big(F \tone (cTx)\big) \cap TA\big| &= \big|\big(TG \tone Tx \big) \cap TA \big|\\
& = \big|T\big((G \tone_T x) \cap A\big)\big| \\
&= \big|(G \tone_T x) \cap A\big| \\
& \geq \al |G| = \al |F|.
\end{align*}
This shows that $\al \leq \dstar_{\tone}(TA)$. Since $\al < \dstar_{\tone_T}(A)$ was arbitrary, $\dstar_{\tone_T}(A) \leq \dstar_{\tone}(TA)$. The same idea works to prove the reverse inequality.

Next, we claim that $T(\Zdz)$ is in $\pws^*(\tone)$ and, if $(\Zdz,\tone)$ is left amenable, in $\density^*(\tone)$. By Lemma \ref{lem:reachestheid}, replacing if necessary $c$ by a multiple of $c$, there exists $w \in \Zdz$ for which $\rone(w) = c\idd$ and $c \Zdz \subseteq T(\Zdz)$. Thus, it suffices to show that $B = \Zdz \setminus c \Zdz$ is in neither $\pws(\tone)$ nor $\density(\tone)$. Note that $w^{-1} \tone B = \emptyset$. It follows by Lemma \ref{lem:basicfactsofdensity} \ref{item:inversetranslatehasmoredensity} and the fact that $w$ is in the center of $(\Zdz,\tone)$ that for all $y_1, \ldots, y_k \in \Zdz$,
\begin{align*}
0 &\leq \dstar_\tone (y_1^{-1} \tone B \cup \cdots \cup y_k^{-1} \tone B) \\
& \leq \dstar_\tone \big(w^{-1} \tone (y_1^{-1} \tone B \cup \cdots \cup y_k^{-1} \tone B) \big)\\
& \leq \dstar_\tone \big (y_1^{-1} \tone (w^{-1} \tone B) \cup \cdots \cup y_k^{-1} \tone (w^{-1} \tone B) \big)\\
& = \dstar_\tone(\emptyset) = 0.
\end{align*}
By the definition of piecewise syndeticity and Lemma \ref{lem:basicfactsofdensity} \ref{item:basicsofthick}, this shows that $B \not\in \pws(\tone)$ and, if $(\Zdz,\tone)$ is left amenable, $B \not\in \density(\tone)$.

By (\ref{eqn:equaldensities}) and Lemma \ref{lem:basicfactsofdensity} \ref{item:basicsofthick}, (\ref{eqn:inclassifandonlyif}) holds immediately for the classes $\density$ and $\thick$. To see that it holds for $\syndetic$, let $A \subseteq \Zdz$ and put $B = \Zdz \setminus A$. Now $A \in \syndetic (\tone_T)$ if and only if $B \not\in \thick(\tone_T)$ if and only if $TB \not\in \thick(\tone)$. By Lemma \ref{lem:basicfactsofduals} \ref{item:psstarintersectedlarge}, since $T(\Zdz) \in \pws^*(\tone)$, $TB \not\in \thick(\tone)$ if and only if $TB \cup (\Zdz \setminus T(\Zdz)) \not\in \thick(\tone)$ if and only if $\Zdz \setminus TA \in \thick(\tone)$ if and only if $TA \in \syndetic (\tone)$.

Since (\ref{eqn:inclassifandonlyif}) holds for $\thick$ and $\syndetic$, it holds for $\pws$ by Lemma \ref{lem:basicfactsofduals} \ref{item:pwsisintersection}. It holds for $\pws^*$ and $\density^*$ by the same reasoning as in the previous paragraph, replacing ``not thick'' with ``zero density'' (using Lemma \ref{lem:basicfactsofdensity} \ref{item:unionoftwozerodensity} and \ref{item:intersectionofdstar}) and ``not piecewise syndetic'' (using Lemma \ref{lem:basicfactsofduals} \ref{item:psstarintersectedlarge}).
\end{proof}


\begin{lemma}\label{lem:specialpropertyofgldqaction}
If $T \in \gldq \cap \matdz$ and $\tone_T \in \{\tone,\toneop\}$, then $T \in \gldz$.
\end{lemma}

\begin{proof}
Suppose $\tone_T = \tone$. We have by equation (\ref{eqn:repofactedonmult}) that for all $x \in \Zd$,
\begin{align}\label{eqn:selfrepisequal}T\rone (x)T^{-1} = \rone(Tx).\end{align}
The map $\conj_T: \roneqd \to \roneqd$ defined by $\conj_T(A) = TAT^{-1}$ is a linear map of the $\Q$-vector space $\roneqd$ with determinant equal to 1. Indeed, if $\roneqd$ is ``vectorized'' by considering its elements as column vectors, then the matrix of $\conj_T$ with respect to the usual basis is given by the Kronecker product $T^{\text{t}} \otimes T^{-1}$. By properties of the Kronecker product, $\det(T^{\text{t}} \otimes T^{-1}) = \det(T^{\text{t}})^d \det(T^{-1})^d = 1.$

Let $\lat = \ronezd$ and $\lat'=\rone(T \Zd)$; both are $d$-dimensional lattices in $\roneqd$, and since $T: \Zd \to \Zd$, $\lat' \subseteq \lat$. To show that $T \in \gldz$, it suffices to show that $\lat=\lat'$. Indeed, since $\rone$ is injective, this will prove that $T$ is surjective. Equation (\ref{eqn:selfrepisequal}) gives that $\conj_T(\lat) = \lat'$. This means $\conj_T: \lat \to \lat$ is an injective $\Z$-linear map. Since $\lat$ is full dimensional in $\roneqd$, the determinant of $\conj_T$ as a map of the lattice $\lat$ is equal to $1$. It follows that $\conj_T$ is a lattice isomorphism, that is, $\conj_T(\lat) = \lat$, whereby $\lat' = \conj_T(\lat) = \lat$.

Since $\cdot_{\text{op}}: \discrims \to \discrims$ commutes with the $\gldq$-action on $\discrims$, if $\tone_T = \toneop$, then $\tone_{T^2} = \tone$. By the work above, $\det(T)^2 = 1$, so $\det T = \pm 1$.
\end{proof}

%

The subspaces $\normshort(\tone)$, $\aut(\tone)$, and $\isom(\tone,\toneop)$ of $\gldq$ appearing in the following corollary were defined at the end of Section \ref{sec:ringreps}.

\begin{corollary}\label{cor:whichtspreserveclasses}
Let $\tone \in \discrims$, and $T \in \matdz \cap \gldq$.
\begin{enumerate}[label=(\Roman*)]
\item\label{item:invcorone} For all $\arbclass \in \{\syndetic,\thick,\pws,\pws^*,\density,\density^*\}$, $T$ preserves the class $\arbclass(\tone)$ if and only if $T \in \normshort(\tone)$.
\item\label{item:invcortwo} For all $\arbclass \in \{\ipclass,\ipclass^*\}$, $T$ preserves the class $\arbclass(\tone)$ if and only if $T \in \aut(\tone)$.
\item\label{item:invcorthree} For all $r \geq 2$ and $\arbclass \in \{\iprclass,\iprclass^*,\ipnaughtclass,\ipnaughtclass^*\}$, $T$ preserves the class $\arbclass(\tone)$ if and only if $T \in \aut(\tone) \cup \isom(\tone,\toneop)$.
\end{enumerate}
\end{corollary}

\begin{proof}
By Lemma \ref{lem:reachestheid}, there exist $c \in \N$ and $w \in \Zdz$ such that $\tone_{cT} = (\tone_{\rone(w)})_T \in \discrims$. To see \ref{item:invcorone}, note that if $T \in \normshort(\tone)$, then so is $cT$. In this case, by Lemma \ref{lem:gldqactionproperties}, $\tone$, $\tone_{\rone(w)}$, and $\tone_{cT}$ are all aligned. By Corollary \ref{cor:ifdirectionofmaintheoremone} and Lemma \ref{lem:semigrouphomomorphism},
\begin{align}\label{eqn:stringofcontainments}T \arbclass(\tone)  = T\arbclass((\tone_{\rone(w)})_T) \subseteq \arbclass(\tone_{\rone(w)}) = \arbclass(\tone).\end{align}
Therefore, $T \arbclass(\tone) \subseteq \arbclass(\tone)$.

If $T \not\in \normshort(\tone)$, then neither is $cT$. By Lemma \ref{lem:gldqactionproperties}, $\tone$ and $\tone_{cT}$ are not aligned, so by Corollary \ref{cor:noclasscontainment}, there exists $A \in \arbclass(\tone) \setminus \arbclass(\tone_{cT})$. By Lemma \ref{lem:semigrouphomomorphism}, since $A \not\in \arbclass(\tone_{cT})$, $TA \not\in \arbclass(\tone_{\rone(w)}) = \arbclass(\tone)$. Therefore, $T \arbclass(\tone) \not\subseteq \arbclass(\tone)$.

We will show \ref{item:invcortwo} and \ref{item:invcorthree} simultaneously. If $T \in \aut(\tone) \cup \isom(\tone,\toneop)$, then by Lemma \ref{lem:gldqactionproperties}, $\tone_T \in \{\tone,\toneop\}$, and by Lemma \ref{lem:specialpropertyofgldqaction}, $T \in \gldz$. In this case, by Lemma \ref{lem:semigrouphomomorphism}, (\ref{eqn:stringofcontainments}) holds for any $\arbclass$. Therefore, $T \arbclass(\tone) \subseteq \arbclass(\tone)$.

If $T \not\in \aut(\tone)$, then $\tone_T \neq \tone$. This means $\tone_{cT} \neq \tone_{\rone(w)}$, so by Corollary \ref{cor:conclusionofmaintheormtwo}, there exists $A \in \ipclass^*(\tone_{\rone(w)}) \setminus \ipclass^*(\tone_{cT})$. Since $A \not\in \ipclass^*(\tone_{cT})$ and, by Lemma \ref{lem:semigrouphomomorphism}, $T$ takes $\ip$ sets to $\ip$ sets, $TA \not\in \ipclass^*(\tone_{c})$. By Lemma \ref{lem:semigrouphomomorphism} with multiplication by $c$, using the same argument, $cA \in \ipclass^*(\tone)$ and $TcA \not\in \ipclass^*(\tone)$. This shows $T \ipclass^*(\tone) \not\subseteq \ipclass^*(\tone)$. The same argument works with $\ipclass^*$ replaced by $\ipclass$.

If $T \not\in \aut(\tone) \cup \isom(\tone,\toneop)$, then $\tone_T \not\in \{\tone,\toneop\}$. Recalling that $\cdot_{\text{op}}$ commutes with the $\gldq$-action on $\discrims$, $\tone_{cT} \not\in \{\tone_{\rone(w)},(\tone_{\rone(w)})_{\text{op}}\}$. The argument now proceeds just as in the preceding paragraph.
\end{proof}

Upon writing the condition in (\ref{eqn:mainconditioncorollaryc}) in terms of representations, this completes the proof of Corollary C from the introduction.

As demonstrated in Section \ref{sec:earlyexamples}, a description of the automorphism group $\aut(\tone)$ allows one in many cases to describe $\normshort(\tone)$ and $\isom(\tone,\toneop)$ explicitly. In these cases, Corollary \ref{cor:whichtspreserveclasses} provides a geometric understanding of many classes of multiplicative largeness.

As a basic example of this, in the notation of Section \ref{sec:earlyexamples}, the transformation $T: (x_1,x_2) \mapsto (x_2,x_1)$ is an element of $\normshort(\tthree {x^2+1})$. Therefore, Corollary \ref{cor:whichtspreserveclasses} gives that the class of multiplicatively \psstar{} sets with respect to the multiplication induced on $\Z^2$ from $\Z[i]$ is preserved under reflection about the line $x_1 = x_2$. The map $T$ does not, however, lie in $\normshort(\tthree {x^2-c})$ for any other $c \in \Z \setminus \{0^2, 1^2, \ldots \}$, meaning that it does not preserve the corresponding class of multiplicatively large sets with respect to any of the other multiplications induced from the rings $\Z[\sqrt{c}]$.

\section{A combinatorial characterization of \texorpdfstring{$\pws^*$}{PS*} and sources of \adive{} \texorpdfstring{$\ipr^*$}{IPr*} sets in \texorpdfstring{$\Zd$}{Zd}}\label{sec:applications} 

Theorem \ref{thm:thmfrombgotherpaper} gives that \adive{} $\ipr^*$ sets in $\Zd$ are multiplicatively \psstar{} with respect to all proper multiplications on $\Zd$. Theorem A and its improvements in Section \ref{sec:manymults} show that the classes of multiplicatively \psstar{} sets for the various proper multiplications on $\Zd$ are, predominantly, in general position. Thus, the results in this paper serve to enhance the conclusions of those results which yield \adive{} $\ipr^*$ sets in $\Zd$. In this section, we give examples of such results from combinatorics and measure theoretical and topological dynamics.

First, we give a combinatorial characterization of sets in the class $\pws^*(\tone)$. The proof of this characterization follows from Lemma \ref{lem:basicfactsofduals} \ref{item:combcharofpwsstar} and the definitions. Recall that $\cube N = \{-N,\ldots,N\}^d \setminus \{0\}$.

\begin{lemma}\label{lem:altcharofpstar}
Let $A \subseteq \Zd$ and $\tone \in \discrims$. $A \in \pws^*(\tone)$ if and only if for all $F \in \finitesubsets {\Zdz}$, there exists $N \in \N$ such that for all $x \in \Zdz$, there exists $z \in \cube N$ such that $F \tone z \tone x \subseteq A$.
\end{lemma}

It is useful to compare this characterization with the combinatorial characterizations for sets in $\thick(\tone)$ and, supposing $(\Zdz,\tone)$ is left amenable, sets in $\density^*(\tone)$:
\begin{align*}
A \in \thick(\tone) &\Leftrightarrow \forall F \in \finitesubsets \Zdz, \ \exists x \in \Zdz, \ F \tone x \subseteq A;\\
A \in \pws^*(\tone) &\Leftrightarrow \forall F \in \finitesubsets \Zdz, \ \exists N \in \N, \ \forall x \in \Zdz, \ \exists z \in \cube N, \ F \tone z \tone x \subseteq A; \\
A \in \density^*(\tone) &\Leftrightarrow \forall \eps > 0, \ \exists F \in \finitesubsets \Zdz, \ \forall x \in \Zdz, \ \big|(F \tone x) \cap A\big| \geq (1-\eps) |F|.
\end{align*}
This shows that sets in the class $\pws^*(\tone)$ can be thought of as multiplicatively ``quantitatively very thick.'' Keep in mind that this property extends by Theorem \ref{thm:thmfrombgotherpaper} to all of the \adive{} $\ipr^*$ sets appearing below.

Times of multiple recurrence of sets of positive density provide a combinatorial source of additive $\ipr^*$ sets. The additive upper Banach density of the set $A \subseteq (\Zd,+)$ can be written as
\[\dstar_+(A) = \sup_{(F_N)_N} \limsup_{N \to \infty} \frac{|A \cap F_N|}{|F_N|} = \limsup_{N \to \infty} \max_{x \in \Zd} \frac{\big| A \cap (x + \{1, \ldots, N\}^d) \big|}{N^d},\]
where the supremum is in the first expression is over all \folner sequences $(F_N)_{N\in\N} \subseteq \finitesubsets \Zd$ (see Remark \ref{rem:folnersequences}). Equality with the second expression follows by Definition \ref{def:density}, Lemma \ref{lem:basicfactsofdensity} \ref{item:dstarandbanachcoincide}, and the fact that for any $B \subseteq \Zd$, there exists a left translation invariant mean $(\Zd,+)$ whose value on the indicator function of $B$ is $\dstar_+(B)$. See \cite[Section 3]{BGpaperone} for a more thorough discussion.

Let $A \subseteq \Z^d$ have positive additive upper Banach density. It is not hard to show that there exists an $r \in \N$ for which the set of times of single recurrence
\[A - A = \{z \in \Z^d \ | \ A \cap (A-z) \neq \emptyset\}\]
is an $\ipr^*$ subset of $\Zd$. In fact, the set $A-A$ is much larger: it is a $\Delta_r^*$ set, meaning that for all $(z_i)_{i=1}^r \subseteq \Z^d$, there exist $i \neq j$ such that $z_j - z_i \in A-A$. 

While the $\Delta_r^*$ property does not hold in general for times of \emph{multiple} recurrence (see \cite[Chapter 9.1]{furstenberg-book}), it was shown in \cite{furstenbergkatznelsonipszem} that such sets are \adive{} $\ipr^*$ sets. When $d=1$ and $T_i = i$ in the following theorem, we recover the $\ipr^*$ formulation of Szemer\'{e}di's theorem mentioned in the introduction.

\begin{theorem}[{\cite[Theorem 10.3]{furstenbergkatznelsonipszem}}]\label{thm:multiszem}
Let $n \in \N$ and $\delta > 0$. There exists $r \in \N$ such that for all $T_1, \ldots, T_n \in \matdz$ and all $A \subseteq \Zd$ with $\dstar_+(A) > \delta$, the set
\[\big\{ z \in \Zd \ \big | \ \text{there exists $x \in \Zd$, for all $i \in \{1, \ldots, n\}$, $x + T_iz \in A$} \big\}\]
is \adly{} $\ipr^*$.
\end{theorem}

Another main source of additive $\ipr^*$ sets arises in topological dynamics and is implicit in the proofs in \cite[Section 1]{bergelsonleibmanams}. We will describe a large class of such sets related to Diophantine approximation and constant-free generalized polynomials; see \cite{BLACTA}. Define inductively a nested sequence $\big(\genpoly_n(\Zd) \big)_{n \in \N}$ of sets of functions $\Z^d \to \R$ by
\begin{align*}
\genpoly_1(\Zd) &= \big\{ x \mapsto c_1 x_1 + \cdots + c_d x_d \ \big| \ c_1, \ldots, c_d \in \R \big\},\\
\genpoly_{n+1}(\Zd) &= \bigcup_{f,g \in \genpoly_{n}(\Zd)} \big\{f+g, f \cdot g, [f] \big\},
\end{align*}
where $[f]$ denotes the function $x \mapsto [f(x)]$, the fractional part of $f(x)$. The set $\genpoly(\Zd) = \cup_{n \geq 1} \genpoly_n(\Zd)$ is the set of \emph{constant-free generalized polynomials on $\Zd$}. In addition to including polynomials in $\R[x_1,\dots,x_d]$ with zero constant term, this set includes (when $d=2$) functions such as $(x_1,x_2) \mapsto \pi x_1 \big[ \sqrt{2} x_2^2 \big]$, $(x_1,x_2) \mapsto e^2(x_1-x_2^3)\big[\sqrt[3] 2 x_1^7 [0.5 x_2] \big]$, and the function described in the set in $(\ref{eqn:impressiveexample})$ inside the $\|\cdot\|$, the distance to the nearest integer function.

\begin{theorem}[{\cite[Theorem 0.4]{BLiprstarcharacterization}}]\label{thm:genpolyshaveiprstarreturns}
Let $n \in \N$ and $\eps > 0$. There exists $r \in \N$ such that for all $f \in \genpoly_n(\Zd)$, the set
\[ \big\{ x \in \Zd \ \big| \ \|f(x)\| < \eps \big\}\]
is \adly{} $\ipr^*$.
\end{theorem}

Since the intersection of $\ipnaught^*$ sets is $\ipnaught^*$ (see \cite[Proposition 2.5]{BRcountablefields}), the conclusion of this theorem holds just as well for functions $\Zd \to \R^m$ which consist of a constant-free generalized polynomial in each coordinate.  The special case of Theorem \ref{thm:genpolyshaveiprstarreturns} for polynomials with zero constant term appears as \cite[Theorem 7.7]{BergelsonSurveytwoten}, where it is proved using the Hales-Jewett theorem. The polynomial Hales-Jewett theorem \cite{bergelsonleibmanannals} is used to prove \cite[Theorem 0.4]{BLiprstarcharacterization} in a much stronger form than is stated in Theorem \ref{thm:genpolyshaveiprstarreturns} above.

Theorem \ref{thm:genpolyshaveiprstarreturns} is a finitary analogue of \cite[Theorem D]{BLACTA}. In that paper, Theorem D was used to prove an $\ip^*$-improvement of a classical result of van der Corput \cite[Satz 11]{vanderCorputtwo}. By the same proof given in \cite{BLACTA}, the following $\ipr^*$-improvement of van der Corput's result follows from Theorem \ref{thm:genpolyshaveiprstarreturns}.

\begin{theorem}[{cf. \cite[Theorem 0.34]{BLACTA}}]
Let $k,n \in \N$ and $\eps > 0$. There exists $r \in \N$ such that for all $f_i \in \genpoly_n(\Z^{d+i-1})$, $i = 1, \ldots, k$, the set of $x \in \Zd$ for which there exist $m_1, \ldots, m_k \in \Z$ satisfying
\[\big|f_1(x) - m_1 \big| < \eps, \ \big|f_2(x,m_1) - m_2 \big| < \eps, \ \ldots, \ \big|f_k(x,m_1,\ldots,m_{k-1}) - m_k \big| < \eps,\]
is \adly{} $\ipr^*$.
\end{theorem}

For further discussion regarding $\ipr^*$ sets and return times in measure theoretical and topological dynamics, the reader is referred to \cite[Section 6]{BGpaperone}.


\bibliographystyle{alphanum}
\bibliography{multizdbib}

\end{document}

%% file: multizd_header.tex
\theoremstyle{plain}
\newtheorem{theorem}{Theorem}[section]
\newtheorem{lemma}[theorem]{Lemma}

\newtheorem{corollary}[theorem]{Corollary}

\newtheorem*{theorem*}{Theorem 0}

\newtheorem*{mainthma}{Theorem A}
\newtheorem*{mainthmb}{Theorem B}

\newtheorem*{maincorc}{Corollary C}
\newtheorem*{claim*}{Claim}

\theoremstyle{definition}
\newtheorem{definition}[theorem]{Definition}
\newtheorem{remark}[theorem]{Remark}

\newcommand{\adly}{additively}
\newcommand{\adive}{additive}

\newcommand{\Adive}{Additive}

\newcommand{\al}{\alpha}
\newcommand{\ga}{\gamma}
\newcommand{\be}{\beta}
\newcommand{\eps}{\epsilon}

\newcommand{\newtran}{\sigma}
\newcommand{\condone}{(1)}
\newcommand{\condtwo}{(2)}
\newcommand{\condthree}{(3)}

\newcommand{\R}{\mathbb{R}}

\newcommand{\Q}{\mathbb{Q}}
\newcommand{\Qz}{\mathbb{Q}_{\neq 0}}

\newcommand{\Qd}{\mathbb{Q}^d}
\newcommand{\Qdz}{\mathbb{Q}_{\neq 0}^d}
\newcommand{\Z}{\mathbb{Z}}
\newcommand{\Zz}{\mathbb{Z}\setminus \{0\}}

\newcommand{\Zt}{\mathbb{Z}^2}
\newcommand{\Ztz}{\mathbb{Z}_{\neq 0}^2}
\newcommand{\Zd}{\mathbb{Z}^d}
\newcommand{\Zdz}{\mathbb{Z}_{\neq 0}^d}

\newcommand{\N}{\mathbb{N}}

\newcommand{\arbclass}{\mathcal{X}}

\newcommand{\syndetic}{\mathcal{S}}
\newcommand{\thick}{\mathcal{T}}
\newcommand{\pws}{\mathcal{PS}}

\newcommand{\central}{\mathcal{C}}
\newcommand{\density}{\mathcal{D}}

\newcommand{\ip}{\text{IP}}
\newcommand{\ipr}{\text{IP}_{r}}
\newcommand{\ipnaught}{\ip_{\text{0}}}

\newcommand{\ipclass}{\mathcal{IP}}
\newcommand{\iprclass}{\mathcal{IP}_{r}}
\newcommand{\ipnaughtclass}{\ipclass_{\text{0}}}

\newcommand{\ipset}{\text{IP}}
\newcommand{\iprset}[1]{\ipset_{#1}}

\newcommand{\finiteproducts}{\text{FP}}
\newcommand{\finitesums}{\text{FS}}

\newcommand{\folner}{F{\o}lner }
\newcommand{\dstar}{d^*}

\newcommand{\euclid}[1]{|{#1}|}

\newcommand{\psstar}{PS$^*$}

\newcommand{\tone}{\circledast}
\newcommand{\toneop}{\tone_{\text{op}}}
\newcommand{\ttwo}{\odot}
\newcommand{\ttwoop}{\odot_{\text{op}}}
\newcommand{\rone}{\psi}

\newcommand{\ronezd}{\rone(\Zd)}
\newcommand{\ronezdz}{\rone(\Zdz)}
\newcommand{\roneqd}{\rone(\Qd)}
\newcommand{\roneqdz}{\rone(\Qdz)}
\newcommand{\roneright}{\psi_\text{r}}
\newcommand{\ronerightqd}{\roneright(\Qd)}

\newcommand{\rtwo}{\varphi}
\newcommand{\rtwozd}{\rtwo(\Zd)}

\newcommand{\rtwoqd}{\rtwo(\Qd)}
\newcommand{\rtworight}{\rtwo_\text{r}}

\newcommand{\tthree}[1]{\tone_{[{#1}]}}
\newcommand{\rthree}[1]{\rone_{[{#1}]}}

\newcommand{\id}{\text{Id}}
\newcommand{\idd}{\text{Id}}

\newcommand{\tgen}{\cdot}

\newcommand{\matdz}{M_d(\Z)}
\newcommand{\matdq}{M_d(\Q)}
\newcommand{\matdR}{M_d(R)}
\newcommand{\gldq}{\text{GL}_d(\Q)}
\newcommand{\gldz}{\text{GL}_d(\Z)}
\newcommand{\gldR}{\text{GL}_d(R)}

\newcommand{\nullspace}{\text{Null}}
\newcommand{\nullcoll}{\mathcal{N}}

\newcommand{\zspan}{\text{Span}_\Z}
\newcommand{\qspan}{\text{Span}_\Q}

\newcommand{\centshort}{\text{Cen}}
\newcommand{\normshort}{\text{Nor}}
\newcommand{\aut}{\text{Aut}}

\newcommand{\isom}{\text{Iso}}

\newcommand{\normmin}[1]{\|#1 \|_{\text{min}}}
\newcommand{\normmax}[1]{\|#1 \|_{\text{max}}}

\newcommand{\discrims}{\mathscr{B}}
\newcommand{\qdiscrims}{\widetilde{\mathscr{B}}}
\newcommand{\finitediscrims}{\mathscr{F}}
\newcommand{\bigmults}{\mathscr{A}}

\newcommand{\lat}{\mathcal{L}}
\newcommand{\conj}{\mathscr{J}}

\newcommand{\genpoly}{\mathscr{G}}

\newcommand{\quat}{\mathbb{H}}
\newcommand{\quatz}{\mathbb{H}(\Z)}

\newcommand{\subsets}[1]{\mathcal{P}({#1})}
\newcommand{\finitesubsets}[1]{\mathcal{P}_f({#1})}
\newcommand{\seminplus}{(\N,+)}

\newcommand{\semistimes}{(S,\tgen)}

\newcommand{\matrixx}[4]{\left( \begin{array}{cc} #1 & #2 \\ #3 & #4 \end{array} \right)}

\newcommand{\cube}[1]{C_{#1}}

\newcommand\restr[2]{{ \left.\kern-\nulldelimiterspace #1 \right|_{#2}}}